\documentclass[11pt]{article}
\usepackage[all,2cell]{xy} \UseAllTwocells \SilentMatrices
\usepackage{latexsym,amsfonts,amssymb}
\usepackage{amsmath,amsthm,amscd}
\usepackage{hyperref,psfrag}
\usepackage{diagcat}
\usepackage{color}
\usepackage{etoolbox}

\usepackage[dvips]{epsfig}
\usepackage{psfrag}

\usepackage{graphicx}

\usepackage{a4wide}

\newstrandstyle{Red}{type=ribbon,style=solid,ribboncolor=lightred,color=red}

\normalfont\upshape

\usepackage{fancyheadings}
\usepackage[all]{xy}
\SelectTips{cm}{}

\theoremstyle{plain}
\newtheorem{thm}{Theorem}[section]
\newtheorem{prop}[thm]{Proposition}
\newtheorem{lemma}[thm]{Lemma}
\newtheorem{cor}[thm]{Corollary}

\theoremstyle{definition}

\newtheorem{defn}[thm]{Definition}
\newtheorem{eg}[thm]{Example}
\newtheorem{rmk}[thm]{Remark}
\newtheorem{ntn}[thm]{Notation}

\numberwithin{equation}{section}

\usepackage{bbm}
\def\1{\mathbbm 1}

\def\Q{\mathbb Q}
\def\Z{\mathbb Z}
\def\N{\mathbb N}

\def\F{\mathbb F}
\def\o{\otimes}
\def\lra{\longrightarrow}

\def\RHOM{\mathbf{R}\mathrm{HOM}}
\def\Id{\mathrm{Id}}
\def\mc{\mathcal}
\def\mf{\mathfrak}
\def\Ext{\mathrm{Ext}}

\newcommand{\udmod}{\!-\!\underline{\mathrm{mod}}}  
\newcommand{\dmod}{\!-\!\mathrm{mod}}

\newcommand{\dif}{\partial}   

\newcommand{\NH}{\mathrm{NH}} 

\newcommand{\Hom}{{\rm Hom}}
\newcommand{\HOM}{{\rm HOM}}

\def\dif{{\partial}}

\def\lra{{\longrightarrow}}
\def\dmod{{\mathrm{-mod}}}   

\def\Ext{{\mathrm{Ext}}}

\def\Id{\mathrm{Id}}
\def\mc{\mathcal}
\def\mf{\mathfrak}
\def\shuffle{\,\raise 1pt\hbox{$\scriptscriptstyle\cup{\mskip
               -4mu}\cup$}\,}

\newcommand{\refequal}[1]{\xy {\ar@{=}^{#1}
(-1,0)*{};(1,0)*{}};
\endxy}

\newstrandstyle{object}{type=ribbon,style=solid,ribboncolor=gray,ribbonbdrycolor=black,width=18pt}
\newstrandstyle{object2}{type=ribbon,style=solid,ribboncolor=blue,ribbonbdrycolor=black,width=18pt}

\title{A Categorification of the Burau representation at prime roots of unity}
\author{You Qi, Joshua Sussan}

\date{December 30, 2013}
\begin{document}
%

\maketitle

\setcounter{tocdepth}{2}

\tableofcontents

\begin{abstract}
We construct a $p$-DG structure on an algebra Koszul dual to a zigzag algebra used by Khovanov and Seidel to construct a categorical braid group action.  We show there is a braid group action in this $p$-DG setting\let\thefootnote\relax\footnotetext{Key words: Braid group action; Burau representation; $p$-DG algebra; hopfological algebra.}. \let\thefootnote\relax\footnotetext{MSC (2010): 81R50, 16E20, 16E35.}
\end{abstract}

\section{Introduction}
Since Khovanov's ground breaking work \cite{KhJones}, significant progress has been made for categorification of link invariants as well as their representation-theoretical explanations \cite{Web,Web2} for a generic value of $ q $.
The Reshetikhin-Turaev and Turaev-Viro quantum three-manifold invariants require $q$ to be specialized to a root of unity.
The corresponding problems of categorically lifting quantum three-manifolds remains mostly all open.
With this in mind, Khovanov~\cite{Hopforoots} introduced the subject of \emph{hopfological algebra} and proposed a categorification program when $q$ equals a prime root of unity.  The key observation of Khovanov was that, over a field of characteristic $p>0$, the homotopy category of $p$-complexes, historically first considered by Mayer \cite{Mayer1,Mayer2}, categorifies the ring of integers in the cyclotomic field $\mathbb{Q}[\zeta_p]$, where $\zeta_p$ is a primitive $p$th root of unity.  Khovanov further suggested that, in order to find interesting such categorifications, one should look for algebras with a nilpotent derivation $ \partial $ of order $ p $.

As a first instance, Khovanov and Qi~\cite{KQ} showed that the nilHecke algebra carries a $p$-nilpotent derivation over a field of characteristic $p$.  Using the technical machinery developed by Qi ~\cite{QYHopf}, they showed that the derived category of compact modules over all nilHecke algebras equipped with this nilpotent derivation categorifies half of quantum $ \mathfrak{sl}_2 $ at a $p$th root of unity.  They also show how to equip KLR algebras with a nilpotent derivation which would conjecturally extend their result to a categorification of the upper half of simply-laced quantum groups at a root of unity of prime order. Another observation of ~\cite{KQ} is that Webster's algebras \cite{Web,Web2,Web3} also admit a family of nilpotent derivations, promising a way to categorically specialize the Witten-Reshetikhin-Turaev tangle invariant at a prime root of unity. The latter process in turn constitutes an important step towards a categorification of the Witten-Reshetikhin-Turaev three-manifold invariant.

In the current work, we give a rather modest justification of the above program by singling out a relatively easy example.  We consider the simplest non-trivial block of Webster's algebra which categorifies the second highest weight space of $V_1^{\otimes n}$, where $V_1$ is the defining representation of quantum $\mathfrak{sl}_2$. The main reason that we would like to present this particular case by itself is that, these rings, or in  a Koszul dual ``disguise,'' appear naturally in various contexts in symplectic topology, algebraic geometry and representation theory, and has attracted much attention in recent years (see, for example, \cite{KS, SeidelThomas, HKh, CauLiSu} and the references therein). We combine these algebras with $p$-nilpotent differentials, and thereby exhibit a (weak) categorification of the Burau representation of the Temperley-Lieb algebra and the braid group at a prime root of unity.

\paragraph{Outline.}We now give a brief summary of the contents of this paper.

We begin by defining various versions of the small quantum $ \mf{sl}_2$ along with tensor powers of the standard representation in Section \ref{sec-small-qgroup}.  The Burau representation is then given explicitly as the braid group action on the second highest weight space.

In Section \ref{generalities}, some general facts about $p$-DG algebras and their hopfological theory are reviewed.  Much of the theory is tailored for our particular example introduced in Section \ref{zigzagalgebra}: a zig-zag algebra with a family of $p$-differentials.

As a common theme in representation theory, (categorical) braid group actions naturally arise from compatible (categorical) quantum group actions. For this reason, we explain in Section \ref{zigzagalgebra} why only two particular differentials, which are conjugate to each other by an automorphism, are relevant for categorification purposes. The hopfological behavior of the other derivations seems more challenging, and we do not have a comparatively good understanding of them.

Finally, a categorification of the Burau representation at a prime root of unity is constructed in Section \ref{catburau} (see Theorems \ref{thm-TL-action} and \ref{thm-braidrelations}).  The main computational tool that enables the hopfological computations and the proof of the main results are the cofibrant replacements of the simple objects (Lemma \ref{lemma-ny-resolution}), which we address fondly as our ``New Year's resolution\footnote{We will abbreviate it as the ``NY resolution'' in the main text.},'' in celebration of the time we arXiv this paper, as well as  motivating ourselves to keep on with categorification at a prime root of unity. In the appendix we construct explicit maps between cofibrant replacements of simple objects.  This may be useful in the future if one wants to investigate the faithfulness of the categorified braid group action at primes roots of unity.

\paragraph{Acknowledgements.}
The authors would like to thank Mikhail Khovanov for valuable suggestions throughout the course of this project and his continued support. The first author thanks his friend Rumen Zarev for many helpful discussions.


\section{The small quantum group}\label{sec-small-qgroup}
\paragraph{The small quantum group.}
Let $ \zeta_{2l}$ be a primitive $2l$th root of unity.  The quantum group $ u_{\zeta_{2l}}(\mathfrak{sl}_2) $, which we will denote simply by $ u_{\zeta_{2l}} $, is generated by $ E, F, K^{\pm 1} $ subject to relations:
\begin{enumerate}
\item $ KK^{-1} = K^{-1}K=1$
\item $ K^{\pm 1}E = \zeta_{2l}^{\pm 2} EK^{\pm 1} $, \quad $ K^{\pm 1} F = \zeta_{2l}^{\mp 2} FK^{\pm 1} $
\item $ EF-FE = \frac{K-K^{-1}}{\zeta_{2l}-\zeta_{2l}^{-1}} $
\item $ E^l = F^l = 0 $
\end{enumerate}

The quantum group is a Hopf algebra whose comultiplication map $ \Delta \colon u_{\zeta_{2l}} \rightarrow u_{\zeta_{2l}} \otimes u_{\zeta_{2l}} $ is given on generators by
\begin{equation*}
\Delta(E) = E \otimes 1 + K \otimes E \quad\quad \Delta(F) = 1 \otimes F + F \otimes K^{-1} \quad\quad \Delta(K^{\pm 1}) = K^{\pm 1} \otimes K^{\pm 1}.
\end{equation*}

Let $ [n] = \frac{\zeta_{2l}^n - \zeta_{2l}^{-n}}{\zeta_{2l}-\zeta_{2l}^{-1}} $, $ E^{(n)} = \frac{E^n}{[n]!} $, and $ F^{(n)} = \frac{F^n}{[n]!} $.
The elements $ E^{(n)}$, $F^{(n)}$, and $ K^n$ generate an algebra over the ring of cyclotomic integers $ \mathcal{O}_{2l} = \mathbb{Z}[\zeta_{2l}] $.
Denote this integral form by $ u_{\zeta_{2l}} $.

Now let $ l=p $ be prime.  Introduce the auxiliary ring $ \mathbb{O}_p = \mathbb{Z}[q]/(\Psi_p(q^2)) $, where $ \Psi_p(q) $ is the $p$-th cyclotomic polynomial.
We can define in a similar fashion an integral form $ u_{\mathbb{O}_p} $ for the small quantum group $\mf{sl_2}$ over $ \mathbb{O}_p $.  For more details see \cite[Section 3.3]{KQ}.
Let the lower half of $ u_{\mathbb{O}_p} $ be the subalgebra generated by the $ F^{(n)} $ and denote it by $ u_{\mathbb{O}_p}^- $

Let $ {V}_n $ be the unique (up to isomorphism) irreducible module (type I) for $ u_{\mathbb{O}_p} $ of rank $ n+1$.  It has a basis $ \lbrace v_0, v_1, \ldots, v_{n} \rbrace $ such that
\begin{equation}
\label{irreddef}
K^{\pm 1} v_i= q^{\pm (n-2i)} v_i\quad\quad Fv_i =[i+1] v_{i+1}\quad\quad E v_i = [n-i+1] v_{i-1}.
\end{equation}

On $ V_1^{\otimes n} $ there is an action of $ u_{\mathbb{O}_p} $ given via the comultiplication map $ \Delta$.  There is also a commuting action of the braid group that factors through the Hecke algebra.  We consider the second highest weight space which we denote by
$ V_1^{\otimes n}[n-2] $.  This is spanned by vectors $ v_{i_1} \otimes \cdots \otimes v_{i_n} $ where $ i_r \in \lbrace 0, 1 \rbrace $ and exactly one $ i_r$ is $ 1 $.
It is convenient to use another basis of $ V_1^{\otimes n}[n-2] $ which is spanned by vectors $ l_r$ where
\begin{equation*}
l_r =
\begin{cases}
v_{i_1} \otimes \cdots \otimes v_{i_{r-1}} \otimes (v_1 \otimes v_0-qv_0\o v_1) \otimes v_{i_{r+2}} \otimes \cdots \otimes v_{i_n}  & \text{ if } i_{r}= 1  \text{ and } r \neq n \\
v_{i_1} \otimes \cdots \otimes v_{i_n} & \text{ if } i_r=1 \text{ and } r=n.
\end{cases}
\end{equation*}

\paragraph{The braid group.}
The braid group $B_n$ is generated by elements $ t_i $ for $i =1, \ldots, n-1 $ subject to the relations that
\begin{enumerate}
\item $t_i t_j= t_j t_i$ if $|i-j|>1$,
\item $ t_i t_{i +1} t_i = t_{i+1} t_i t_{i+1} $ for $ i=1,\ldots,n-2$.
\end{enumerate}
For $i=1,\ldots, n-1$, it will be convenient to introduce the inverse of $t_i$ and denote it by $t_i^\prime$.

Define an operator $ \tilde{t} $ on $ V_1^{\otimes 2} $ by
\begin{align*}
\tilde{t}(v_0 \otimes v_0) &= v_0 \otimes v_0 \\
\tilde{t}(v_1 \otimes v_1) &= v_1 \otimes v_1 \\
\tilde{t}(v_0 \otimes v_1) &= q v_1 \otimes v_0 +(1-q^2) v_0 \otimes v_1 \\
\tilde{t}(v_1 \otimes v_0) &= q v_0 \otimes v_1.
\end{align*}
Then there is an action of the braid group generator $ {t}_i \colon V_1^{\otimes n} \rightarrow V_1^{\otimes n} $ by
\begin{equation*}
t_i = \Id^{\otimes i-1} \otimes \tilde{t} \otimes \Id^{\otimes n-i-1}.
\end{equation*}
Define an operator $ \tilde{t}' $ on $ V_1^{\otimes 2} $ by
\begin{align*}
\tilde{t}'(v_0 \otimes v_0) &= v_0 \otimes v_0 \\
\tilde{t}'(v_1 \otimes v_1) &= v_1 \otimes v_1 \\
\tilde{t}'(v_0 \otimes v_1) &= q^{-1} v_1 \otimes v_0 \\
\tilde{t}'(v_1 \otimes v_0) &=  (1-q^{-2})v_1 \otimes v_0 + q^{-1} v_0 \otimes v_1.
\end{align*}
Then there is an action of the braid group generator $ {t}_i' \colon V_1^{\otimes n} \rightarrow V_1^{\otimes n} $ by
\begin{equation*}
t_i' = \Id^{\otimes i-1} \otimes \tilde{t}' \otimes \Id^{\otimes n-i-1}.
\end{equation*}
It is easy to see that these operators preserve each weight space of $V_1^{\otimes n} $ and it is straightforward to check that they do indeed satisfy the braid relations.
We call the action of $ B_n $ on $ V_1^{\otimes n}[n-2] $ the Burau representation.

In terms of the basis $ \{l_1, \ldots, l_n \} $ the action of the braid group is given by
\begin{equation}
\label{twistingdualcan1}
t_i(l_j) =
\begin{cases}
-q^2 l_i & \text{ if } i=j \\
q l_i + l_j & \text{ if } |i-j|=1 \\
l_j & \text{ if } |i-j|>1
\end{cases}
\end{equation}
\begin{equation}
\label{twistingdualcan2}
t_i'(l_j) =
\begin{cases}
-q^{-2} l_i & \text{ if } i=j \\
q^{-1} l_i + l_j & \text{ if } |i-j|=1 \\
l_j & \text{ if } |i-j|>1.
\end{cases}
\end{equation}


\section{Generalities on \texorpdfstring{$p$}{p}-DG algebras}
\label{generalities}
\paragraph{Elements.} We gather in this section some necessary facts about $p$-DG theory that we shall use.  Fix once and for all a base field $\Bbbk$ of characteristic $p > 0$. Unadorned tensor product $\o$ will stand for the tensor product over $\Bbbk$. Throughout we use the French convention $\N:=\{0,1,2,\dots\}$. Set $H:=\Bbbk[\dif]/(\dif^p)$ which is a graded Hopf algebra with $\deg(\dif)=2$.

A $p$-DG algebra $(A,\dif_A)$ over a ground field $\Bbbk$ of positive characteristic $p$ consists of a $\Z$-graded algebra $A=\oplus_{i\in \Z}A^i$ and a degree-two endomorphism $\dif_A$ on $A$, which is $p$-nilpotent and satisfies the Leibnitz rule:
\[
\dif_A^p\equiv 0,  \ \ \ \ \dif_A(ab)=\dif_A(a)b+a\dif_A(b),
\]
for any $a,b\in A$. A left (resp. right) $p$-DG module $(M,\dif_M)$ over $(A,\dif_A)$ is a left (resp. right) $A$-module equipped with a degree two, $p$-nilpotent endomorphism $\dif_M$, $\dif_M^p\equiv 0$, which is compatible with the differential action on $A$:
\[
\dif_M(ax)=\dif_A(a)x+a\dif_M(x) \ \ \ \ (\textrm{resp.}~\dif_M(xa)=\dif_M(x)a+x\dif_A(a)).
\]
For a $p$-DG module $ (M,\dif_M)$ we will also consider the shifted module $ (M\{l\},\dif_M) $ where the degree $i$ component $ M\{l\}^i$ is equal to $M^{i-l}$.
In what follows, we will be exclusively considering $p$-DG algebras of the following type.
\begin{defn}\label{defn-positive-p-DGA}A $p$-DG algebra $A$ is called \emph{positive}, if it satisfies the following three conditions.
\begin{itemize}
\item[(i)] The algebra $A$ is non-negatively graded, $A\cong \oplus_{i\in \N}A^i$, and is finite dimensional in each fixed degree.
\item[(ii)] The homogeneous degree-zero part $A^0$ is semi-simple.
\item[(iii)] The differential $\dif_A$ acts trivially on $A^0$.
\end{itemize}
\end{defn}

The subscripts in the differentials on the $p$-DG algebras and modules will be dropped when no confusion can be caused. We adapt some basic definitions from \cite[Section 2]{KQ} in the situation when $A$ is positive, and the reader is referred loc.~cit.~for a summary of the basic definitions and techniques of the homological theory of $p$-DG modules. Some further information on the subject in the context of \emph{hopfological algebra} can be found in \cite{QYHopf}. In particular, we will use the notations $A_\dif\dmod$ (resp.~$\mc{C}(A)$, $\mc{D}(A)$) to stand for the abelian (resp.~homotopy, derived) category of $p$-DG modules over $A$. Here $A_\dif$ denotes the \emph{smash product algebra}\footnote{It is usually denoted by $A \# H$.}, which is isomorphic, as a vector space, to $A\o H$, subject to the relations that $A\cong A\o 1\subset A\o H$, $H\cong  1\o H \subset A\o H$ as subalgebras, and the rule for commuting elements
\[
(1\o \dif) \cdot (a\o 1 )=(a\o 1) \cdot(1\o \dif) + \dif(a)\o 1,
\]
where $a$ is an arbitrary element of $A$. The categories $\mc{C}(A)$, $\mc{D}(A)$ are triangulated, equipped with the \emph{shift} endo-functors which we denote by $[1]$. Recall that for any $p$-DG module $M$, the shift functor $[1]$ of $\mc{C}(A)$ or $\mc{D}(A)$ acts on $M$ by $M[1]\cong M\o \widetilde{V}_{p-2}\{-p\}$, where $\widetilde{V}_{p-2}\{-p\}$ is the $(p-1)$-dimensional $p$-complex
\[
0\lra \underbrace{\Bbbk = \Bbbk = \cdots =\Bbbk}_{p-1~\textrm{copies}} \lra 0.
\]
Here the left most $\Bbbk$ sits in $q$-degree $-2p+2$ while the right most $\Bbbk$ sits in $q$-degree $-2$. The quasi-inverse of $[1]$, denoted $[-1]$, is given by the tensor product with the $p$-complex $\widetilde{V}_{p-2}\{p\}$. An easy exercise shows that shifting twice in the triangulated category is equivalent to a grading shift by $-2p$, i.e., there is an isomorphism
\begin{equation}\label{eqn-hom-shift-two-equals-grading-shift}
M[2]\cong M\{-2p\}
\end{equation}
which is functorial in $M$. Homological grading shifts $[\pm 1]$ give rise to a categorical $\mathbb{Z}$-action on the triangulated categories $\mc{C}(A)$ and $\mc{D}(A)$ by fixing, once and for all, a quasi-isomorphism
\begin{equation}\label{eqn-canonical-iota}
\iota \colon \widetilde{V}_0\lra \widetilde{V}_{p-2}\o \widetilde{V}_{p-2}, \quad \iota(\tilde{u}_0) =\sum_{i=0}^{p-2}(-1)^{i} \tilde{v}_{i}\o \tilde{v}_{p-2-i}.
\end{equation}
Here $\tilde{u}_0$ and $\tilde{v}_i$ denote some fixed balanced basis vectors for $\widetilde{V}_0$ and $\widetilde{V}_{p-2}$, such that $\dif(\tilde{u}_0)=0$, $\dif(\tilde{v}_i)=\tilde{v}_{i+1}$ ($0\leq i \leq p-3$) and $\dif(\tilde{v}_{p-2})=0$.

As a matter of convention, given two graded modules $M$, $N$ over a graded algebra $A$, we set $\Hom_{A}(M,N)$ to be the space of degree preserving $A$-module maps between $M$ and $N$. We define
$$\HOM_{A}(M,N):=\oplus_{l\in \Z}\Hom_{A}(M,N\{l\}),$$
where $N\{l\}$ stands for the same underlying $A$-module $N$ with gradings shifted up by $l$.

\begin{defn}\label{def-nice-p-dg-module}Let $A$ be a $p$-DG algebra, and $M$ be a $p$-DG module over $A$.
\begin{itemize}
\item[(i)] The module $M$ is called \emph{cofibrant} if for any surjective quasi-isomorphism of $p$-DG modules $N_1\twoheadrightarrow N_2$, the induced map
\[
\HOM_A(M,N_1)\lra \HOM_A(M,N_2)
\]
of $p$-complexes is a homotopy equivalence.
\item[(ii)] When $A$ is positive, we say $M$ is a \emph{finite cell module} if it has a finite exhaustive increasing filtration $F^\bullet$ by $p$-DG submodules such that the subquotients
\[
F^\bullet/F^{\bullet-1}\cong P_i\{l_i\}
\]
for some $i=1,\dots ,n$ and $l_i\in \Z$ where $P_i$ is a projective $A$-module.
\end{itemize}
\end{defn}

Cofibrant modules are ``nice'' in the sense that we have a good control of the morphism spaces between them in the derived category $\mc{D}(A)$. Namely, there is, for any cofibrant module $M$, a functorial-in-$N$ isomorphism of $\Bbbk$-vector spaces
\[
\Hom_{\mc{C}(A)}(M,N)\cong \Hom_{\mc{D}(A)}(M,N),
\]
which is induced from the localization functor $\mc{C}(A)\lra \mc{D}(A)$. It is an easy exercise to show that finite cell modules are cofibrant. Furthermore, it is clear that if $V$ is a finite dimensional $p$-complex $V$, then $A\o V$ is a finite cell module. We regard finite cell modules in the $p$-DG setting as an analogue of the ``one-sided twisted complexes'' in the usual DG theory \cite{BondalKapranov}.

\begin{eg}\label{eg-finite-cofibrant-mod}Let $A$ be a $p$-DG algebra. Assume that $a_0$ is an element of $A$ with $\dif_A^{p-1}(a_0)=0$, and $a_0\not\in\mathrm{Im}{\dif_A}$. Here we construct an example of a finite cell module that is not of the form $A\o V$.

Consider a free $A$-module on two generators $M:=Av_0\oplus Av_1$, where $\mathrm{deg}(v_0)-\mathrm{deg}(v_1)=2-\mathrm{deg}(a_0)$. Define a differential $\dif_M$ as follows. On generators, it acts by
\[
\dif_M(v_0)=0, \ \ \ \ \dif_M(v_1)=a_0v_0.
\]
It is extended to all of $M$ by the Leibnitz rule. We leave it as an exercise for the reader to check that $\dif_M^p\equiv 0$, and $M$ is not isomorphic to the tensor product of $A$ by a two-dimensional $p$-complex.
\end{eg}

We also recall the definition of the $\RHOM$-complex (see~\cite[Definition 8.11]{QYHopf}) in the $p$-DG setting.

\begin{defn}\label{def-RHOM}
Let $A$ be a $p$-DG algebra, and $M$, $N$ be two $p$-DG modules over $A$. The $\RHOM$-complex between $M$ and $N$, denoted $\RHOM_A(M,N)$, is a $p$-complex of the form
\[
\RHOM_A(M,N):=\HOM_A(\mathbf{p}M,N),
\]
where $\mathbf{p}M$ is a cofibrant $p$-DG module that surjects quasi-isomorphically onto $M$.
\end{defn}
It is easy to see that the isomorphism class of the $p$-complex $\RHOM_A(M,N)$, in the homotopy category of $p$-complexes, is independent of the choices of the cofibrant replacement $\mathbf{p}M$ for $M$ (c.~f.~the discussion about resolutions below). The morphism space $\Hom_{\mc{D}(A)}(M,N)$ between two $p$-DG modules $M,N\in \mc{D}(A)$ is canonically isomorphic to the ``stable invariants'' of the $p$-complex $\RHOM_{A}(M,N)$
\begin{equation}\label{eqn-morphism-space-as-invariants}
\Hom_{\mc{D}(A)}(M,N) \cong \Hom_{H\udmod}(\widetilde{V}_0, \RHOM_{A}(M,N)),
\end{equation}
where $\widetilde{V}_0$ stands for the one-dimensional $p$-complex $\Bbbk$ sitting in degree zero. More generally, attached to any two objects in a triangulated category one has the ``$\Ext$-space'' as follows.

\begin{defn}\label{def-ext-in-DA} Given any $p$-DG modules $M,N\in \mc{D}(A)$ and $i\in \Z$, we set the \emph{$i$th $p$-DG $\Ext$-space}, or simply the \emph{$\Ext$-space} between them, to be
\[
\Ext_{\mc{D}(A)}^i (M,N):=\Hom_{\mc{D}(A)}(M,N[i]),
\]
and we define \emph{the total $\Ext$-space} as
\[
\Ext^\bullet_{\mc{D}(A)}(M,N):=\bigoplus_{i\in \Z}\Ext^i_{\mc{D}(A)}(M,N).
\]
\end{defn}
One can also make these definitions in the homotopy category $\mc{C}(A)$ but we will not need them. It follows from \eqref{eqn-morphism-space-as-invariants} that there is a bi-functorial isomorphism
\begin{equation}\label{eqn-ext-space-as-invariants}
\Ext^\bullet_{\mc{D}(A)}(M,N)\cong \bigoplus_{i\in \Z}\Hom_{H\udmod}(\widetilde{V}_0,\RHOM_{A}(M,N)[i]).
\end{equation}

\begin{rmk}\label{rmk-why-not-EXT}
Despite the fact that one recovers an ordinary triangulated category $\mc{D}(A)$ by considering only the $\Ext$-morphism spaces, doing so will result in loss of information on the graded-$H$-module enhancement structure. For a simple instance, let $A=\Bbbk$, so that $\mc{D}(A)\cong H\udmod$. Consider the modules $M\cong \widetilde{V}_0$, $N\cong\widetilde{V}_0\{2\}$, so that
\[
\RHOM_\Bbbk(M,N)\cong \widetilde{V}_0\{2\}, \quad \quad \quad \RHOM_{\Bbbk}(N,M)\cong \widetilde{V}_0\{-2\}.
\]
Here, we have used that $M$, $N$ are always cofibrant as a $p$-DG module over $\Bbbk$, so that
$$\RHOM_{\Bbbk}(M,N)\cong \HOM_{\Bbbk}(M,N)\cong M^*\otimes_\Bbbk N.$$
It follows from \eqref{eqn-ext-space-as-invariants} that, on the one hand,
\[
\Ext^{i}_{\mc{D}(\Bbbk)}(M,N) \cong
\left\{
\begin{array}{cc}
\Bbbk & i=1,\\
0 & i \neq 1;
\end{array}
\right.
\]
while on the other hand, for degree reasons, $\Ext_{\mc{D}(\Bbbk)}^\bullet(N,M)\cong 0$, breaking in the apparent symmetry from grading shifts.

By equation \eqref{eqn-ext-space-as-invariants}, the $\RHOM$-space completely determines the $\Ext$-space. Henceforth, we will mostly just consider the former.
\end{rmk}

We next recall the notion of compact modules, which takes place in the derived category.

\begin{defn} \label{def-compact-mod} Let $A$ be a $p$-DG algebra. A $p$-DG module $M$ over $A$ is called \emph{compact} (in the derived category $\mc{D}(A)$) if and only if, for any family of $p$-DG modules $N_i$ where $i$ takes value in some index set $I$, the natural map
\[
\oplus_{i\in I}\Hom_{\mc{D}(A)}(M,N_i)\lra \Hom_{\mc{D}(A)}(M,\oplus_{i\in I} N_i)
\]
is an isomorphism of $\Bbbk$-vector spaces.
\end{defn}

The strictly full subcategory of $\mc{D}(A)$ consisting of compact modules will be denoted by $\mc{D}^c(A)$.  It is triangulated and will be referred to as the \emph{compact derived category}.

For positive $p$-DG algebras, we have a close relationship between finite cell modules and compact modules. Denote by $\mc{F}(A)$ the collection of all finite cell modules over $A$, and let $\overline{\mc{F}}(A)$ be the closure of $\mc{F}(A)$ inside $\mc{D}(A)$ under isomorphisms. Notice that the natural inclusion functor $\overline{\mc{F}}(A)\subset \mc{D}(A)$ factors through the compact derived category $\mc{D}^c(A)$.

\begin{prop}\label{prop-characterization-compactness}The natural inclusion functor
$$\overline{\mc{F}}(A)\subset \mc{D}^c(A)$$
is an equivalence of triangulated categories. \hfill$\square$
\end{prop}

In other words, over a positive $p$-DG algebra, any compact module is quasi-isomorphic to a finite cell module. The proof of the Proposition can be found in \cite[Theorem 2.17]{EQ1}. Applying the usual Grothendieck group construction to the triangulated category $\mc{D}^c(A)$, we obtain the $p$-DG Grothendieck group of a $p$-DG algebra $A$, which we denote by $K_0(A)$. It has a natural module structure over $K_0(H\udmod)\cong \mathbb{O}_p$, arising from the categorical action of $H\udmod$ on $\mc{D}^c(A)$.

\paragraph{Resolutions.} Let $M$ be a $p$-DG module. By a \emph{$p$-DG resolution} of $M$ we mean a quasi-isomorphism of $p$-DG modules $P\lra M$ such that $P$ is cofibrant. It is known that resolutions for $p$-DG modules always exist \cite[Theorem 6.6]{QYHopf}, usually not uniquely. However, it is not true in general that any $p$-DG module has a resolution by a finite cell module. Proposition \ref{prop-characterization-compactness} says that, over a positive $p$-DG algebra, this happens precisely when the module is compact in the derived category.

\begin{ntn}\label{ntn-finite-cell-mod}In what follows, we introduce a convenient way to rewrite finite cell modules. Let $M$ be a $p$-DG module with a finite increasing filtration $F^\bullet$ ($\bullet\in \Z$), with $F^i/F^{i-1}:= P_i$. Then we will write this filtered module schematically as a direct sum of its associated graded, and put an arrow from $P_i$ to $P_j$ decorated by a map $\phi$ if $\dif_M(p_i)$ projects onto $\phi(p_i)\in P_j$ (necessarily $j < i$), where $p_i$ is a lift of an element of $P_i$ inside $M$. For instance, assuming $n\geq m$ are integers, the following diagram stands for a filtered $p$-DG module
\[
\xymatrix{0\ar[r] & P_n \ar[r]^-{\phi_n} & P_{n-1}\ar[r]^-{\phi_{n-1}} \ar[r] & \cdots\ar[r]^-{\phi_{m+2}} & P_{m+1} \ar[r]^-{\phi_{m+1}}& P_m \ar[r]& 0}
\]
with the filtration implicitly understood as $F^{i}=0$ if $i\leq m-1$, $F^i=\sum_{k=m}^iP_k$ if $m\leq i\leq n$, while $F^i$ is the entire module if $i\geq n$. On each $F^i$, the differential structure is defined inductively by the composition
\[
\dif_{F^i}: P_j \lra P_j+P_{j-1}\subset F^i, \ \ \ \ p_j\mapsto (\dif_{P_j}(p_j), \phi_j(p_j)),
\]
for all $j \in \{m, m+1,\dots, i\}$.
\end{ntn}

Alternatively, one can regard the original filtration on $M$ as obtained from collapsing the bigrading of a bigraded $A$-module $M\cong \bigoplus_{i,j\in \Z}P_j^i$ into a single grading, which will always be referred to as the \emph{$q$-grading}, and take $F^j:=\sum_{k\geq j, i\in \Z}P^i_k$. Note that $\dif_M$ does not necessarily preserve the bigrading.

For instance, in this notation, the filtered $p$-DG module in Example \ref{eg-finite-cofibrant-mod} will be represented in this notation by
\[
0\lra A \stackrel{\cdot a_0}{\lra} A \lra 0.
\]

For another example, given any $p$-DG module $M$, its homological shift $M[1]$ will be depicted by the filtered module
\[
0\lra \underbrace{M = M = \cdots =M}_{p-1~\textrm{copies}} \lra 0.
\]
More generally, if $V_i$ denotes the $(i+1)$-dimensional indecomposable $p$-complex, then $M\o V_i$ will be written as
\[
0\lra \underbrace{M = M = \cdots =M}_{i+1~\textrm{copies}} \lra 0.
\]

Using this notation, we give a way of describing the (co)cone construction, in analogue with the usual DG case.

\begin{lemma}\label{lemma-cone-construction}
Let $\phi:M\lra N$ be a map of $p$-DG modules over $A$. The cone of the morphism $\phi$ is isomorphic to the filtered $p$-DG module
\[
\xymatrix{0 \ar[r] & M  \ar@{=}[r] & \cdots \ar@{=}[r] & M\ar[r]^{-\phi}& N \ar[r]& 0,}
\]
where $N$ sits in $q$-degree zero, and $M$ is repeated $p-1$ times. The cocone of $\phi$ is isomorphic to the filtered module
\[
\xymatrix{0 \ar[r] & M  \ar[r]^{-\phi} & N \ar@{=}[r] & \cdots \ar@{=}[r] & N \ar[r]& 0,}
\]
where $M$ sits in $q$-degree zero, and $N$ is repeated $p-1$ times.
\end{lemma}

For later convenience, and in agreement with the usual DG case, we will simply write the (co)cone of $\phi$ as
\[
M[1] \xrightarrow{\phi[-1]} N \quad \quad \left(\textrm{resp.}~M \xrightarrow{\phi[-1]}N[-1]\right).
\]

\begin{proof}
This follows from the definition of the cone and cocone. See \cite[Definition 2.14]{KQ}.
\end{proof}

\begin{lemma}\label{lemma-ses-gives-p-DG-resolution} Let $0\lra K \stackrel{\phi}{\lra} L  \stackrel{\psi}{\lra} M \lra 0$ be a short exact sequence of graded modules over the smash product algebra $A_\dif$. Then the filtered $p$-DG modules
\begin{subequations}
\begin{equation}\label{eqn-ses-lead-to-acyclic-pdg-mod-1}
0\lra K\stackrel{\phi}{\lra}\underbrace{L=\cdots =L}_{p-1}  \stackrel{\psi}{\lra} M\lra 0
\end{equation}
\begin{equation}\label{eqn-ses-lead-to-acyclic-pdg-mod-2}
0\lra \underbrace{K=\cdots =K}_{p-1} \stackrel{\phi}{\lra} L \stackrel{\psi}{\lra}  \underbrace{M=\cdots =M}_{p-1}\lra 0
\end{equation}
\end{subequations}
are acyclic.
\end{lemma}
\begin{proof}In fact, \eqref{eqn-ses-lead-to-acyclic-pdg-mod-1} differs from \eqref{eqn-ses-lead-to-acyclic-pdg-mod-2} only by a shift functor up to acyclic summands. It suffices to show either one of them.

The original short exact sequence, treated as a filtered $p$-DG module in the above notation, admits a $p$-DG submodule which is filtered isomorphic to
$$0\lra K \stackrel{\cong}{\lra} \mathrm{Ker}(\psi) \lra 0.$$
Thus we have a new short exact sequence of filtered $p$-DG modules:
\[
\xymatrix{0 \ar[r] & K \ar[r]^-{\cong} \ar@{=}[d] & \mathrm{Ker}(\psi)\ar[r] \ar[d] & 0\ar[r]\ar[d] & 0 \\
0 \ar[r] & K \ar[r]\ar[d] & L \ar[r]\ar[d] & M \ar[r]\ar@{=}[d] & 0\\
0\ar[r]& 0 \ar[r] & \mathrm{Coker}(\phi)\ar[r]^-{\cong} & M\ar[r] & 0 ,&}
\]
where we treat the first row as a submodule, whilst the last row as the quotient. In this manner, \eqref{eqn-ses-lead-to-acyclic-pdg-mod-1} with the middle term extended $p-1$ many times fits into the following diagram:
\[
\xymatrix{0 \ar[r] & K \ar[r] \ar@{=}[d] & \mathrm{Ker}(\psi)\ar@{=}[r]\ar[d]&\cdots\ar@{=}[r]& \mathrm{Ker}(\psi) \ar[r]\ar[d] & 0\ar[r]\ar[d] & 0 \\
0 \ar[r] & K \ar[r]\ar[d] & L\ar@{=}[r]\ar[d] & \cdots \ar@{=}[r] &  L \ar[r]\ar[d] & M \ar[r]\ar@{=}[d] & 0\\
0\ar[r]& 0\ar[r] &\mathrm{Coker}(\phi)\ar@{=}[r]& \cdots \ar@{=}[r]&\mathrm{Coker}(\phi)\ar[r] & M \ar[r] & 0 .&}
\]

The top and bottom filtered modules are obviously acyclic $p$-complexes, hence so is the middle one. The lemma follows.
\end{proof}

\begin{rmk}[Grading shifts]\label{rmk-deg-shifts}
We make a remark about the gradings involved in the discussion about resolutions. By convention, the maps $\phi$, $\psi$ in the above short exact sequences are homogeneous degree-zero maps of graded $A_\dif$-modules. When building a filtered $p$-DG module out of the short exact sequence, we need the differentials to be homogeneous of degree two. Therefore we have to shift the $q$-degrees of $L$'s and $M$'s accordingly. There is always a unique way to do this up to an overall shift. For instance, in the first filtered module in the above lemma, if we fix the $q$-degree of $M$ to be zero, then the corresponding degree shifts are given as below.
\[0\lra K\{-2p\}\stackrel{\phi}{\lra} L\{-2p+2\}=\cdots =L\{-2\}  \stackrel{\psi}{\lra} M\lra 0.\]
Unless we need to stress them, the $q$-degree shifts will be left unspecified.
\end{rmk}

\begin{cor}\label{cor-unbounded-complexes}Let
\[
\xymatrix{\cdots \ar[r]^-{\phi_{m+4}} & P_{m+3} \ar[r]^-{\phi_{m+3}} & P_{m+2}\ar[r]^-{\phi_{m+2}} & P_{m+1} \ar[r]^-{\phi_{m+1}}& P_m \ar[r]& 0}
\]
be a bounded-from-above exact sequence of graded modules over the smash product algebra $A_\dif$. Then the filtered $p$-DG module, with every other term repeated $(p-1)$-times,
\[
\xymatrix{\cdots \ar[r]^-{\phi_{m+4}} & P_{m+3} \ar@{=}[r]& \cdots \ar@{=}[r] & P_{m+3}\ar[r]^-{\phi_{m+3}} \ar[r] &
P_{m+2}\ar `r[rd] `_l `[llld] _-{\phi_{m+2}} `[d] [lld]
& \\
& & P_{m+1}\ar@{=}[r]&\cdots \ar@{=}[r]& P_{m+1} \ar[r]^-{\phi_{m+1}}& P_m \ar[r]& 0}
\]
is acyclic. \hfill$\square$
\end{cor}

\paragraph{Hopfologically finite algebras.}
We end this section with a brief discussion of \emph{hopfologically finite} $p$-DG algebras. Throughout, for simplicity, we will assume that $A$ is a finite-dimensional $p$-DG algebra.

\begin{defn}\label{def-hopfological-finite-algebra}The $p$-DG algebra $A$ is called \emph{left (resp.~right) hopfologically finite} if all simple left (resp.~right) modules over $A_\dif$, the smash product algebra, are compact in the derived category of left (resp.~right) $p$-DG modules.
\end{defn}

Given a finite-dimensional $p$-DG algebra $A$, we have a strictly full subcategory $\mc{D}^f(A)\subset \mc{D}(A)$ which consists of $p$-DG left (resp.~right) modules that are quasi-isomorphic to finite-dimensional left (resp.~right) $A_\dif$-modules. See \cite[Section 9]{QYHopf} for more details. The natural inclusion of $\mc{D}^c(A)$ into $\mc{D}(A)$ factors through $\mc{D}^f(A)$, since $A$ is finite-dimensional.

\begin{lemma}\label{lemma-equivalence-DC-DF}If $A$ is left (resp.~right) hopfologically finite, then the inclusion functor $\mc{D}^c(A)\lra \mc{D}^f(A)$ is an equivalence of triangulated categories of left (resp.~right) $p$-DG modules.
\end{lemma}
\begin{proof} It suffices to show that any finite dimensional $p$-DG module $M$ over $A$ is compact. Consider $M$ as a module over $A_\dif$ and take a Jordan-H\"{o}lder filtration of $M$. The filtration leads to a convolution diagram (Postnikov tower) of distinguished triangles in $\mc{D}(A)$, which presents $M$ as iterated cones of maps of simple left (resp.~right) $A_\dif$-modules in the derived category. By our assumption, each simple module is compact, then so is $M$ by the usual ``two-out-of-three'' property (see, for instance, \cite[Lemma 7.2]{QYHopf}).
\end{proof}

In the rest of the paper, we will be exclusively focusing on a family of $p$-DG algebras that are positive (Definition \ref{defn-positive-p-DGA}).  Now let $A$ be a positive $p$-DG algebra, and $\{e_i|i=1,\dots, n\}$ be a complete list of pairwise non-isomorphic indecomposable idempotents in $A$. Define $P_i$ (resp. $_iP$) to be the $p$-DG module $A\cdot e_i$ (resp. $e_i\cdot A$), and $L_i:=P_i/A^{\prime}\cdot P_i$ (resp. $_iL:={_iP}/{_iP\cdot A^{\prime}}$), where $A^{\prime}=\oplus_{i>0}A^i$ is the augmentation ideal of $A$. Then $A$ is left (resp. right) hopfologically finite if and only if all $L_i$'s (resp. $_iL$'s), $i=1,\dots,n$ are compact in the derived category. Furthermore, the natural categorical pairing, which exists for any finite dimensional $p$-DG algebra,
\begin{equation}\label{eqn-categorical-pairing}
\RHOM_A(-,-) : \mc{D}^c(A) \times \mc{D}^f(A) \lra H\udmod
\end{equation}
gives rise to a pairing on $\mc{D}^c(A)$ via the equivalence $\mc{D}^c(A)\lra \mc{D}^f(A)$ of Lemma \ref{lemma-equivalence-DC-DF}. It descends to a semi-linear non-degenerate pairing on the Grothendieck group
\begin{equation}\label{eqn-K-group-pairing}
[\RHOM_A(-,-)] : K_0(A) \times K_0(A) \lra \mathbb{O}_p,
\end{equation}
which is perfect if all $L_i$'s ($i=1,\dots, n$) are absolutely irreducible.

%


\section{The zig-zag algebra \texorpdfstring{$A_n^{!}$}{An} and differentials}
\label{zigzagalgebra}
\paragraph{The quiver presentation.} Let $\Bbbk$ be a field of any characteristic $p\geq 0$ (later we will only focus on the case $p>0$). Let $n$ be a natural number greater or equal to two, and $Q_n$ be the following quiver:
\begin{equation}\label{quiver-Q}
\xymatrix{
 \overset{1}{\circ} &
 \cdots \ltwocell{'}&
 \overset{i-1}{~\circ~}\ltwocell{'}&
 \overset{i}{\circ}\ltwocell{'}&
 \overset{i+1}{~\circ~}\ltwocell{'}&
 \cdots \ltwocell{'}&
 \overset{n}{\circ}\ltwocell{'}
 }
\end{equation}
Let $\Bbbk Q_n$ be the path algebra associated to $Q_n$ over the ground field. Following Khovanov-Seidel \cite{KS}, we use, for instance, the symbol $(i|j|k)$, where $i,j,k$ are vertices of the quiver $Q_n$, to denote the path which starts at a vertex $i$, then goes through $j$ (necessarily $j=i\pm 1$) and ends at $k$.
The composition of paths is given defined by
\begin{equation}
(i_i|i_2|\cdots|i_r)\cdot (j_1|j_2|\cdots|j_s)=
\left\{
\begin{array}{ll}
(i_i|i_2|\cdots|i_r|j_2|\cdots|j_s) & \textrm{if $i_r=j_1$,}\\
0 & \textrm{otherwise,}
\end{array}
\right.
\end{equation}
where $i_1,\dots, i_r$ and $j_1, \dots, j_s$ are sequences of neighboring vertices in $Q_n$.

\begin{defn}\label{def-algebra-An}The algebra $A_n^!$ is the quotient of the path algebra $\Bbbk Q_n$ by the relations
\[
(i|i-1|i)=(i|i+1|i)~~ (i=2,\dots, n-1),~~ (1|2|1)=0.
\]
\end{defn}

The algebra $A_n^!$ is Koszul, whose quadratic dual is isomorphic to the algebra $A_n$ considered in \cite{KS}.


\paragraph{Relation to Webster algebras.}
The algebra $A_n^!$ is a particular block of a family of algebras that categorify the $n$-fold tensor product of the standard $U_q(\mf{sl}(2))$ representation. We begin by recalling the definition of a particular Webster algebra. More general versions of these algebras, which are associated with arbitrary finite Lie types, can be found in \cite{Web,Web2,Web3}.

The Webster algebra $W$ in this case is a local algebra with $n$ red strands of width one and any number of black strands. The locality indicates that generators on far away strands commute with each other. The black strands are allowed to carry dots, and red strands are not allowed to cross each other. We depict the local generators of this algebra by
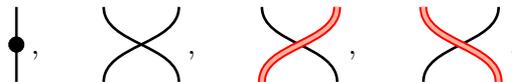
\[
\begin{DGCpicture}
\DGCstrand(0,0)(0,1)
\DGCdot{0.5}
\end{DGCpicture}
\ , \quad \quad
\begin{DGCpicture}
\DGCstrand(1,0)(0,1)
\DGCstrand(0,0)(1,1)
\end{DGCpicture}
\ , \quad \quad
\begin{DGCpicture}
\DGCstrand(1,0)(0,1)
\DGCstrand[Red](0,0)(1,1)
\end{DGCpicture}
\ , \quad \quad
\begin{DGCpicture}
\DGCstrand(0,0)(1,1)
\DGCstrand[Red](1,0)(0,1)
\end{DGCpicture}
\ .
\]
The relations between the local generators are given by the usual nilHecke algebra relations among black strands
\begin{gather}
\begin{DGCpicture}[scale=0.55]
\DGCstrand(1,0)(0,1)(1,2)
\DGCstrand(0,0)(1,1)(0,2)
\end{DGCpicture}
~= 0 \,  \quad \quad
\begin{DGCpicture}[scale=0.55]
\DGCstrand(0,0)(2,2)
\DGCstrand(1,0)(0,1)(1,2)
\DGCstrand(2,0)(0,2)
\end{DGCpicture}
~=~
\begin{DGCpicture}[scale=0.55]
\DGCstrand(0,0)(2,2)
\DGCstrand(1,0)(2,1)(1,2)
\DGCstrand(2,0)(0,2)
\end{DGCpicture}
\ , \\
\begin{DGCpicture}
\DGCstrand(0,0)(1,1)
\DGCdot{0.25}
\DGCstrand(1,0)(0,1)
\end{DGCpicture}
-
\begin{DGCpicture}
\DGCstrand(0,0)(1,1)
\DGCdot{0.75}
\DGCstrand(1,0)(0,1)
\end{DGCpicture}
~=~
\begin{DGCpicture}
\DGCstrand(0,0)(0,1)
\DGCstrand(1,0)(1,)
\end{DGCpicture}
~=~
\begin{DGCpicture}
\DGCstrand(1,0)(0,1)
\DGCdot{0.75}
\DGCstrand(0,0)(1,1)
\end{DGCpicture}
-
\begin{DGCpicture}
\DGCstrand(1,0)(0,1)
\DGCdot{0.25}
\DGCstrand(0,0)(1,1)
\end{DGCpicture} \ ,
\end{gather}
the local relations among red-black strands
\begin{gather}
\begin{DGCpicture}[scale=0.55]
\DGCstrand(1,0)(0,1)(1,2)
\DGCstrand[Red](0,0)(1,1)(0,2)
\end{DGCpicture}
~=~
\begin{DGCpicture}[scale=0.55]
\DGCstrand(1,0)(1,2)
\DGCdot{1}
\DGCstrand[Red](0,0)(0,2)
\end{DGCpicture}
\ , \quad \quad \quad
\begin{DGCpicture}[scale=0.55]
\DGCstrand(0,0)(1,1)(0,2)
\DGCstrand[Red](1,0)(0,1)(1,2)
\end{DGCpicture}
~=~
\begin{DGCpicture}[scale=0.55]
\DGCstrand(0,0)(0,2)
\DGCdot{1}
\DGCstrand[Red](1,0)(1,2)
\end{DGCpicture}
\ , \\
\begin{DGCpicture}[scale=0.5]
\DGCstrand(0,0)(2,2)
\DGCstrand(1,0)(0,1)(1,2)
\DGCstrand[Red](2,0)(0,2)
\end{DGCpicture}
~=~
\begin{DGCpicture}[scale=0.5]
\DGCstrand(0,0)(2,2)
\DGCstrand(1,0)(2,1)(1,2)
\DGCstrand[Red](2,0)(0,2)
\end{DGCpicture}
\ , \quad \quad
\begin{DGCpicture}[scale=0.5]
\DGCstrand(1,0)(0,1)(1,2)
\DGCstrand(2,0)(0,2)
\DGCstrand[Red](0,0)(2,2)
\end{DGCpicture}
~=~
\begin{DGCpicture}[scale=0.5]
\DGCstrand(1,0)(2,1)(1,2)
\DGCstrand(2,0)(0,2)
\DGCstrand[Red](0,0)(2,2)
\end{DGCpicture}
\ , \\
\begin{DGCpicture}
\DGCstrand(0,0)(1,1)
\DGCdot{0.25}
\DGCstrand[Red](1,0)(0,1)
\end{DGCpicture}
~=~
\begin{DGCpicture}
\DGCstrand(0,0)(1,1)
\DGCdot{0.75}
\DGCstrand[Red](1,0)(0,1)
\end{DGCpicture}
\ , \quad \quad
\begin{DGCpicture}
\DGCstrand(1,0)(0,1)
\DGCdot{0.25}
\DGCstrand[Red](0,0)(1,1)
\end{DGCpicture}
~=~
\begin{DGCpicture}
\DGCstrand(1,0)(0,1)
\DGCdot{0.75}
\DGCstrand[Red](0,0)(1,1)
\end{DGCpicture}
, \\
\begin{DGCpicture}[scale=0.5]
\DGCstrand(1,0)(3,2)
\DGCstrand(3,0)(1,2)
\DGCstrand[Red](2,0)(1,1)(2,2)
\end{DGCpicture}
\ - \
\begin{DGCpicture}[scale=0.5]
\DGCstrand(1,0)(3,2)
\DGCstrand(3,0)(1,2)
\DGCstrand[Red](2,0)(3,1)(2,2)
\end{DGCpicture}
\ = \
\begin{DGCpicture}[scale=0.5]
\DGCstrand(1,0)(1,2)
\DGCstrand(3,0)(3,2)
\DGCstrand[Red](2,0)(2,2)
\end{DGCpicture}
\ ,
\end{gather}
together with the \emph{cyclotomic relation} that a black strand, appearing on the far left of any diagram, annihilates the entire picture:
\begin{equation}\label{eqn-cyclotomic}
\begin{DGCpicture}
\DGCstrand(1,0)(1,1)
\DGCcoupon*(1.25,0.25)(1.75,0.75){$\cdots$}
\end{DGCpicture}
~=~0.
\end{equation}

Of particular interest to us in the current work is the block of $W$ where there are $n$ red strands and only one black strand. We identify the idempotent in this block, where the $i+1$st strand is black, with the idempotent $(i)\in A_n^!$:
\[
(i)=
\begin{DGCpicture}
\DGCstrand[Red](0,0)(0,1)
\DGCstrand[Red](1,0)(1,1)
\DGCstrand(1.5,0)(1.5,1)[$^{i+1}$`{\ }]
\DGCstrand[Red](2,0)(2,1)
\DGCstrand[Red](3,0)(3,1)
\DGCcoupon*(0.25,0.25)(0.75,0.75){$\cdots$}
\DGCcoupon*(2.25,0.25)(2.75,0.75){$\cdots$}
\end{DGCpicture}
\ ,
\]
for any $2 \leq i\leq  n$
\[
(i-1|i)=
\begin{DGCpicture}
\DGCstrand[Red](0,0)(0,1)
\DGCstrand(1.5,0)(1,1)[$^{i+1}$`{\ }]
\DGCstrand[Red](1,0)(1.5,1)
\DGCstrand[Red](2,0)(2,1)
\DGCstrand[Red](3,0)(3,1)
\DGCcoupon*(0.25,0.25)(0.75,0.75){$\cdots$}
\DGCcoupon*(2.25,0.25)(2.75,0.75){$\cdots$}
\end{DGCpicture}
\ .
\]
and for any $1 \leq i\leq  n-1$
\[
(i+1|i)=
\begin{DGCpicture}
\DGCstrand[Red](0,0)(0,1)
\DGCstrand[Red](1,0)(1,1)
\DGCstrand(1.5,0)(2,1)[$^{i+1}$`{\ }]
\DGCstrand[Red](2,0)(1.5,1)
\DGCstrand[Red](3,0)(3,1)
\DGCcoupon*(0.25,0.25)(0.75,0.75){$\cdots$}
\DGCcoupon*(2.25,0.25)(2.75,0.75){$\cdots$}
\end{DGCpicture}
\ .
\]

One checks easily that this assignment is an algebra isomorphism between the $A_n^!$ and this particular Webster algebra. For instance, the relation $(1|2|1)=0$ is identified with
\[
\begin{DGCpicture}[scale=0.75]
\DGCstrand[Red](0,0)(0,2)
\DGCstrand(0.5,0)(1,1)(0.5,2)
\DGCstrand[Red](1,0)(0.5,1)(1,2)
\DGCstrand[Red](2,0)(2,2)
\DGCcoupon*(1.25,0.5)(1.75,1.5){$\cdots$}
\end{DGCpicture}
~=~
\begin{DGCpicture}[scale=0.75]
\DGCstrand[Red](0,0)(0,2)
\DGCstrand(0.5,0)(0.5,2)\DGCdot{1}
\DGCstrand[Red](1,0)(1,2)
\DGCstrand[Red](2,0)(2,2)
\DGCcoupon*(1.25,0.5)(1.75,1.5){$\cdots$}
\end{DGCpicture}
~=~0.
\]

There is also a variant of the Webster algebra $\widetilde{W}$ without the cyclotomic relation \eqref{eqn-cyclotomic} (see \cite{Web}). One particular feature of these algebras is that black strands can carry any number of dots without killing the diagram. These algebras categorify, as $U^-:=U_q^{-}(\mf{sl}_2)$-modules, the free module $U^-\o V^{\o (r-1)}$, where $r$ is the number of red strands. The categorical $U^-$-action can be described, in a nutshell, by adding or removing black strands to or from the diagrams.

\paragraph{Differentials on Webster algebras.}
In \cite{KQ}, a family of local $p$-nilpotent differential is defined on the Webster algebra. In our particular case, for any $\lambda\in \F_p$,  $\dif_{\lambda}$ acts on the local generators of the Webster algebra by
\begin{gather}
\dif_{\lambda}\left(
\begin{DGCpicture}
\DGCstrand(0,0)(0,1)
\DGCdot{0.5}
\end{DGCpicture}
\right)=
\begin{DGCpicture}
\DGCstrand(0,0)(0,1)
\DGCdot{0.5}[r]{$^2$}
\end{DGCpicture}
\ ,
\quad \quad
\dif_{\lambda}\left(
\begin{DGCpicture}
\DGCstrand(1,0)(0,1)
\DGCstrand(0,0)(1,1)
\end{DGCpicture}
\right)=~
\begin{DGCpicture}
\DGCstrand(1,0)(1,1)
\DGCstrand(0,0)(0,1)
\end{DGCpicture}
-2~
\begin{DGCpicture}
\DGCstrand(1,0)(0,1)
\DGCdot{0.75}
\DGCstrand(0,0)(1,1)
\end{DGCpicture},
\\
\dif_{\lambda}\left(
\begin{DGCpicture}
\DGCstrand(1,0)(0,1)
\DGCstrand[Red](0,0)(1,1)
\end{DGCpicture}
\right)=\lambda
\begin{DGCpicture}
\DGCstrand(1,0)(0,1)
\DGCdot{0.75}
\DGCstrand[Red](0,0)(1,1)
\end{DGCpicture}, \quad \quad
\dif_{\lambda}\left(
\begin{DGCpicture}
\DGCstrand(0,0)(1,1)
\DGCstrand[Red](1,0)(0,1)
\end{DGCpicture}
\right)=(1-\lambda)
\begin{DGCpicture}
\DGCstrand(0,0)(1,1)
\DGCdot{0.75}
\DGCstrand[Red](1,0)(0,1)
\end{DGCpicture}
\ .
\end{gather}
\begin{rmk}
This family of differentials arises naturally by considering $\widetilde{W}$ as endomorphisms of certain $p$-DG polynomial representations, or alternatively as non-trivial classes in $\mathrm{HH}^1(\widetilde{W})$, the first Hochschild cohomology group of $\widetilde{W}$.
\end{rmk}

The quiver algebra $A_n^!$ inherits the differential $\dif_\lambda$, given on the path generators by
\[
\dif_{\lambda}(i|i+1)=\lambda(i|i+1|i|i+1), \quad \quad \dif_{\lambda}(i|i-1)=(1-\lambda)(i|i-1|i|i-1),
\]
and extended to the whole algebra via the Leibnitz rule. Sometimes for convenience, we will write the loop $(i|i+1|i)=(i|i-1|i)$ at the vertex $i$ simply as $c_i$, so that $\dif_\lambda(i|i+1)=\lambda c_i(i|i+1)=\lambda (i|i+1)c_{i+1}$.

Our main goal in this section is that to show that by imposing a compatible categorical $u^-_{\mathbb{O}_p}$-action on this Webster algebra, the choice of parameters reduces significantly.
The method sketched below will be very much analogous to the one used in \cite[Section 4.2]{KQ} about categorical quantum Serre relations in the $p$-DG setting, and we refer the readers to loc.~cit.~for more related details.

We start by looking at the case $r=2$, with two black strands. The Webster algebra in this case contains the following three idempotents.
\[
\begin{DGCpicture}
\DGCstrand[Red](0,0)(0,1)
\DGCstrand(0.5,0)(0.5,1)
\DGCstrand[Red](1,0)(1,1)
\DGCstrand(1.5,0)(1.5,1)
\end{DGCpicture}
\ , \quad \quad
\begin{DGCpicture}
\DGCstrand[Red](0,0)(0,1)
\DGCstrand[Red](0.5,0)(0.5,1)
\DGCstrand(1,0)(1,1)
\DGCstrand(1.5,0)(1.5,1)
\end{DGCpicture}
\ , \quad \quad
\begin{DGCpicture}
\DGCstrand[Red](0,0)(0,1)
\DGCstrand(0.5,0)(0.5,1)
\DGCstrand(1,0)(1,1)
\DGCstrand[Red](1.5,0)(1.5,1)
\end{DGCpicture}
\ .
\]
The corresponding projective (left) modules, depicted as
\[
W_0:=
\begin{DGCpicture}[scale=0.75]
\DGCstrand[Red](0,0)(0,1)
\DGCstrand(1,0)(1,1)
\DGCstrand[Red](2,0)(2,1)
\DGCstrand(3,0)(3,1)
\DGCcoupon(-0.15,1)(3.15,2){$\widetilde{W}$}
\end{DGCpicture} \ ,
\quad \quad
W_1:=
\begin{DGCpicture}[scale=0.75]
\DGCstrand[Red](0,0)(0,1)
\DGCstrand[Red](1,0)(1,1)
\DGCstrand(2,0)(2,1)
\DGCstrand(3,0)(3,1)
\DGCcoupon(-0.15,1)(3.15,2){$\widetilde{W}$}
\end{DGCpicture} \ ,
\quad \quad
W_2:=
\begin{DGCpicture}[scale=0.75]
\DGCstrand[Red](0,0)(0,1)
\DGCstrand(1,0)(1,1)
\DGCstrand(2,0)(2,1)
\DGCstrand[Red](3,0)(3,1)
\DGCcoupon(-0.15,1)(3.15,2){$\widetilde{W}$}
\end{DGCpicture} \ ,
\]
categorify respectively the elements in the tensor-product module $U^-\o V$
\[
F(F(u_0)\o v_0) , \quad F^{2}(u_0\o v_0), \quad F^2(u_0)\o v_0,
\]
where $u_0$ (resp. $v_0$) denotes a highest weight vector in $U^-$ (resp. $V$). An easy computation shows that these elements are not linearly independent over $\Q(q)$, but instead satisfy the relation
\begin{equation}\label{eqn-Serre-relation-K0}
(q+q^{-1})F(F(u_0)\o v_0)=F^2(u_0\o v_0) + F^2(u_0)\o v_0.
\end{equation}
Dividing both sides of \eqref{eqn-Serre-relation-K0} by $q+q^{-1}$, we obtain the \emph{divided power relation}
\begin{equation}\label{eqn-divided-power}
F(F(u_0)\o v_0)=F^{(2)}(u_0\o v_0) + F^{(2)}(u_0)\o v_0,
\end{equation}
where $F^{(2)}:= \frac{F^{2}}{q+q^{-1}}=\frac{F^{2}}{[2]!}$.

Equation \eqref{eqn-divided-power} has a categorical explanation via representation theory of Webster algebras. The elements $F^{(2)}(u_0\o v_0), F(F(u_0)\o v_0), F^{(2)}(u_0)\o v_0$ are respectively categorified by the projective modules over $\widetilde{W}$
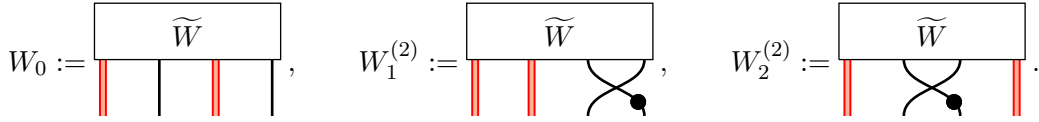
\[
W_0:=
\begin{DGCpicture}[scale=0.75]
\DGCstrand[Red](0,0)(0,1)
\DGCstrand(1,0)(1,1)
\DGCstrand[Red](2,0)(2,1)
\DGCstrand(3,0)(3,1)
\DGCcoupon(-0.15,1)(3.15,2){$\widetilde{W}$}
\end{DGCpicture} \ ,
\quad \quad
W_1^{(2)}:=
\begin{DGCpicture}[scale=0.75]
\DGCstrand[Red](0,0)(0,1)
\DGCstrand[Red](1,0)(1,1)
\DGCstrand(2,0)(3,1)
\DGCstrand(3,0)(2,1)
\DGCdot{0.25}
\DGCcoupon(-0.15,1)(3.15,2){$\widetilde{W}$}
\end{DGCpicture} \ ,
\quad \quad
W_2^{(2)}:=
\begin{DGCpicture}[scale=0.75]
\DGCstrand[Red](0,0)(0,1)
\DGCstrand(1,0)(2,1)
\DGCstrand(2,0)(1,1)
\DGCdot{0.25}
\DGCstrand[Red](3,0)(3,1)
\DGCcoupon(-0.15,1)(3.15,2){$\widetilde{W}$}
\end{DGCpicture} \ .
\]
The last two modules $W_1^{(2)}$, $W_2^{(2)}$ are projective since
$-
\begin{DGCpicture}[scale=0.5]
\DGCstrand(0,0)(1,1)
\DGCstrand(1,0)(0,1)
\DGCdot{0.25}
\end{DGCpicture}
$
is an idempotent in $\NH_2$.

There is an isomorphism (left) $\widetilde{W}$-modules
\begin{equation}\label{eqn-iso-webster-mods}
W_0\cong W_1^{(2)}\oplus W_2^{(2)},
\end{equation}
realized via the following diagram.
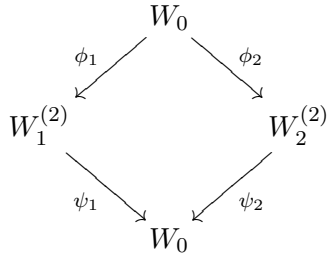
\[
\xymatrix{
& W_0
\ar[dl]_{\phi_1} \ar[dr]^{\phi_2} & \\
W_1^{(2)}\ar[dr]_{\psi_1} && W_2^{(2)}\ar[dl]^{\psi_2}\\
& W_0 &
}
\]
Here $\phi_1$, $\phi_2$ are given by right-multiplication on $W_0$ the elements, which are also denoted by the same letters,
\[
\phi_1=
\begin{DGCpicture}[scale=0.5]
\DGCstrand[Red](0,0)(0,2)
\DGCstrand(2,0)(3,2)
\DGCstrand(3,0)(1,2)
\DGCdot{0.35}
\DGCstrand[Red](1,0)(2,2)
\end{DGCpicture}
 \ , \quad \quad
\phi_2=
\begin{DGCpicture}[scale=0.5]
\DGCstrand[Red](0,0)(0,2)
\DGCstrand(1,0)(3,2)
\DGCstrand(2,0)(1,2)
\DGCdot{0.35}
\DGCstrand[Red](3,0)(2,2)
\end{DGCpicture};
\]
while the maps $\psi_1$, $\psi_2$ are given by
\[
\psi_1= - \
\begin{DGCpicture}[scale=0.5]
\DGCstrand[Red](0,0)(0,2)
\DGCstrand(1,0)(3,2)
\DGCstrand(3,0)(2,2)
\DGCstrand[Red](2,0)(1,2)
\end{DGCpicture}
 \ , \quad \quad
\psi_2=
\begin{DGCpicture}[scale=0.5]
\DGCstrand[Red](0,0)(0,2)
\DGCstrand(1,0)(2,2)
\DGCstrand(3,0)(1,2)
\DGCstrand[Red](2,0)(3,2)
\end{DGCpicture}.
\]
The decomposition \eqref{eqn-iso-webster-mods} is then encapsulated by nil-Hecke relations and the relation in $\widetilde{W}$
\begin{equation}
\begin{DGCpicture}[scale=0.5]
\DGCstrand[Red](0,0)(0,2)
\DGCstrand(1,0)(3,2)
\DGCstrand(3,0)(1,2)
\DGCstrand[Red](2,0)(1,1)(2,2)
\end{DGCpicture}
\ - \
\begin{DGCpicture}[scale=0.5]
\DGCstrand[Red](0,0)(0,2)
\DGCstrand(1,0)(3,2)
\DGCstrand(3,0)(1,2)
\DGCstrand[Red](2,0)(3,1)(2,2)
\end{DGCpicture}
\ = \
\begin{DGCpicture}[scale=0.5]
\DGCstrand[Red](0,0)(0,2)
\DGCstrand(1,0)(1,2)
\DGCstrand(3,0)(3,2)
\DGCstrand[Red](2,0)(2,2)
\end{DGCpicture}
\ .
\end{equation}
Furthermore, the relations
$$F^2(u_0\o v_0)=(q+q^{-1})F^{(2)}(u_0 \o v_0), \quad \quad F^{2}(u_0)\o v_0 = (q+q^{-1})F^{(2)}(u_0)\o v_0$$
are explained by the isomorphism of graded $\widetilde{W}$-modules
\[
W_{i}\cong W_{i}^{(2)}\{1\} \oplus W_{i}^{(2)}\{-1\} \quad \quad (i=1,2).
\]

We do not expect the decomposition \eqref{eqn-iso-webster-mods} to be compatible with the differential $\dif_\lambda$ and thus not to hold directly in the derived category $\mc{D}(\widetilde{W},\dif_{\lambda})$. Instead, we only impose the condition \eqref{eqn-divided-power} or its variant \eqref{eqn-Serre-relation-K0} to hold on the level of Grothendieck groups. Namely, we would like to consider whether the equality of symbols
\begin{equation}\label{eqn-equality-symbols}
(q+q^{-1}){[W_0]}={[W_1]}+{[W_2]}
\end{equation}
holds in $K_0(\widetilde{W}, \dif_\lambda)$. Since
\[
\dif_{\lambda}\left(
\begin{DGCpicture}
\DGCstrand[Red](0,0)(0,1)
\DGCstrand(0.5,0)(0.5,1)
\DGCstrand[Red](1,0)(1,1)
\DGCstrand(1.5,0)(1.5,1)
\end{DGCpicture}\right)
\ = \
\dif_{\lambda}\left(
\begin{DGCpicture}
\DGCstrand[Red](0,0)(0,1)
\DGCstrand[Red](0.5,0)(0.5,1)
\DGCstrand(1,0)(1,1)
\DGCstrand(1.5,0)(1.5,1)
\end{DGCpicture}
\right)
\ = \
\dif_{\lambda}\left(
\begin{DGCpicture}
\DGCstrand[Red](0,0)(0,1)
\DGCstrand(0.5,0)(0.5,1)
\DGCstrand(1,0)(1,1)
\DGCstrand[Red](1.5,0)(1.5,1)
\end{DGCpicture}
\right)
\ = 0
\ ,
\]
the modules $W_0$, $W_1$ and $W_2$ are direct summands of $(\widetilde{W},\dif_{\lambda})$, and thus are cofibrant. If equation \eqref{eqn-equality-symbols}
did hold in $K_0(\widetilde{W},\dif_{\lambda})$, we must have
\begin{equation}\label{eqn-lambda-constraint}
(q+q^{-1})\RHOM_{\widetilde{W}}(W_i,W_0)=\RHOM_{\widetilde{W}}(W_i,W_1)+\RHOM_{\widetilde{W}}(W_i,W_2)
\end{equation}
for $i=0,1,2$. Note that, since each $W_i$ is cofibrant, the $\RHOM$-space may be identified with the usual $\HOM$ space of $\widetilde{W}$-modules.  We now compute the terms in these constraint equations.
\begin{eg} We show by one example how the $p$-complexes $\RHOM_{\widetilde{W}}(W_i,W_j)$, $i,j\in \{0,1,2\}$ are computed. Take $i=0$ and $j=1$ for instance. The space $\HOM_{\widetilde{W}}(W_0,W_1)$ has a $\Bbbk$-basis
\[
\left\{
\begin{DGCpicture}[scale=0.5]
\DGCstrand[Red](0,0)(0,2)
\DGCstrand(3,0)(3,2)
\DGCdot{1.75}[r]{$_b$}
\DGCstrand(2,0)(1,2)
\DGCdot{1.75}[l]{$_a$}
\DGCstrand[Red](1,0)(2,2)
\end{DGCpicture} \ , \ \
\begin{DGCpicture}[scale=0.5]
\DGCstrand[Red](0,0)(0,2)
\DGCstrand(3,0)(1,2)
\DGCdot{1.75}[l]{$_a$}
\DGCstrand(2,0)(3,2)
\DGCdot{1.75}[r]{$_b$}
\DGCstrand[Red](1,0)(2,2)
\end{DGCpicture} \
\Bigg|
a,b\in \N
\right\}
\]

Equipped with $\dif_\lambda$, the $p$-complex $\RHOM_{\widetilde{W}}(W_0,W_1)$ fits into a short exact sequence
\[
0 \lra
\left\langle
\begin{DGCpicture}[scale=0.5]
\DGCstrand[Red](0,0)(0,2)
\DGCstrand(3,0)(3,2)
\DGCdot{1.75}[r]{$_b$}
\DGCstrand(2,0)(1,2)
\DGCdot{1.75}[l]{$_a$}
\DGCstrand[Red](1,0)(2,2)
\end{DGCpicture}\
\Bigg|
a,b\in \N
\right\rangle \lra
\RHOM_{\widetilde{W}}(W_0,W_1)
\lra
\left\langle
\begin{DGCpicture}[scale=0.5]
\DGCstrand[Red](0,0)(0,2)
\DGCstrand(3,0)(1,2)
\DGCdot{1.75}[l]{$_a$}
\DGCstrand(2,0)(3,2)
\DGCdot{1.75}[r]{$_b$}
\DGCstrand[Red](1,0)(2,2)
\end{DGCpicture} \
\Bigg|
a,b\in \N
\right\rangle
\lra  0.
\]
Here angled brackets mean that we are taking the $\Bbbk$-linear span of the elements in between them.
The differential $\dif_\lambda$ acts respectively on the polynomial generators of the sub and quotient complexes by
\[
\dif_\lambda
\left(
\begin{DGCpicture}[scale=0.5]
\DGCstrand[Red](0,0)(0,2)
\DGCstrand(3,0)(3,2)
\DGCstrand(2,0)(1,2)
\DGCstrand[Red](1,0)(2,2)
\end{DGCpicture}
\right)
=
\lambda \
\begin{DGCpicture}[scale=0.5]
\DGCstrand[Red](0,0)(0,2)
\DGCstrand(3,0)(3,2)
\DGCstrand(2,0)(1,2)
\DGCdot{1.75}
\DGCstrand[Red](1,0)(2,2)
\end{DGCpicture}
\ ,
\]

\[
\dif_\lambda\left(
\begin{DGCpicture}[scale=0.5]
\DGCstrand[Red](0,0)(0,2)
\DGCstrand(3,0)(1,2)
\DGCstrand(2,0)(3,2)
\DGCstrand[Red](1,0)(2,2)
\end{DGCpicture}
\right) \equiv (\lambda-2) \
\begin{DGCpicture}[scale=0.5]
\DGCstrand[Red](0,0)(0,2)
\DGCstrand(3,0)(1,2)
\DGCstrand(2,0)(3,2)
\DGCstrand[Red](1,0)(2,2)
\end{DGCpicture} \quad
\left(\mathrm{mod}
\left\{
\begin{DGCpicture}[scale=0.5]
\DGCstrand[Red](0,0)(0,2)
\DGCstrand(3,0)(3,2)
\DGCdot{1.75}[r]{$_b$}
\DGCstrand(2,0)(1,2)
\DGCdot{1.75}[l]{$_a$}
\DGCstrand[Red](1,0)(2,2)
\end{DGCpicture}\
\Bigg|
a,b\in \N
\right\}
\right)
\ .
\]
From these equations, one readily computes the symbol of the $p$-complex $\RHOM_{\widetilde{W}}(W_0,W_1)$, in the Grothendieck ring of $H\udmod$, to be
$$[\RHOM_{\widetilde{W}}(W_0,W_1)]=q(1+q^2+\cdots + q^{2(p-\lambda)})+q^{-1}(1+q^2+\cdots+q^{2(p+2-\lambda)}).$$
\end{eg}

After computing each term in the constraint equations \eqref{eqn-lambda-constraint}, we make a reduction of coefficients along the natural ring map
$$\mathbb{O}_p\cong \Z[q]/(\Psi_{p}(q^2))\twoheadrightarrow \F_p$$
where $q$ on the left-hand side is sent to one under the surjection. The resulted groups of equations read as follows
\begin{gather}
2(1+\lambda(3-\lambda))=(\lambda+3\lambda)+((1-\lambda)+(3-\lambda)),\\
2((1-\lambda)+(3-\lambda))= 4 + ((1-\lambda)^2+(1-\lambda)(3-\lambda)),\\
2(\lambda+3\lambda)=(\lambda^2+(\lambda+2)\lambda)+4.
\end{gather}

Solving these equations together, we obtain $\lambda=1$ aside from the case that $\mathrm{char}(\Bbbk)=2$. We have thus established the ``only if'' part of the following result. (When $\mathrm{char}(\Bbbk)=2$, both values $\lambda=0,1$ are allowed, see Remark \ref{rmk-specialization} below.)

\begin{prop}\label{prop-specialization-lambda}The equality of symbols
$$(q+q^{-1})[W_0]=[W_1]+[W_2]$$
holds in the Grothendieck group of the $p$-DG algebra $(\widetilde{W},\dif_\lambda)$ if and only if $\lambda=1$. Furthermore, when $\lambda=1$, there is a distinguished triangle in $\mc{D}(\widetilde{W},\dif_1)$
\[
\xymatrix{
W_2^{(2)}\ar[r]^{\psi_2} & W_0 \ar[r]^{\phi_1} & W_1^{(2)} \ar[r] & W_2^{(2)}[1].
}
\]
\end{prop}
\begin{proof} It suffices to check the last statement, since we already know that
\[
[W_1]=(q+q^{-1})[W_1^{(2)}], \quad \quad [W_2]=(q+q^{-1})[W_2^{(2)}].
\]
This now follows by showing that the maps $\psi_2$ on $W_2^{(2)}$, $\phi_1$ on $W_0$ intertwines differentials, which we leave to the reader as an exercise.
\end{proof}

\begin{rmk}\label{rmk-specialization}In the discussion of this section, we employed the $p$-DG algebra $(\NH,\dif_1)$ from \cite{KQ}. A similar consideration for the nilHecke algebra equipped with the differential ``$\dif_{-1}$'' of loc.~cit.~shows that Proposition \ref{prop-specialization-lambda} also holds with a different choice of parameters $\lambda=0$. Of course, these choices of parameters for the nilHecke and Webster algebras are related via the algebra anti-automorphism $\tau$ given by reflecting a diagram about a horizontal axis.
\[
\tau\left(
\begin{DGCpicture}[scale=0.5]
\DGCstrand[Red](0,0)(0,2)
\DGCstrand(3,0)(1,2)\DGCdot{0.5}
\DGCstrand(2,0)(3,2)
\DGCstrand[Red](1,0)(2,2)
\end{DGCpicture}
\right)  =
\begin{DGCpicture}[scale=0.5]
\DGCstrand[Red](0,0)(0,2)
\DGCstrand(1,0)(3,2)\DGCdot{1.5}
\DGCstrand(3,0)(2,2)
\DGCstrand[Red](2,0)(1,2)
\end{DGCpicture}
\]
However, for Webster algebras, this anti-automorphism switches the left module category with the right module category over $W$, and changes the parameter in the differential
\begin{equation}\label{eqn-parameter-intertwiner}
\tau\circ \dif_{\lambda}=\dif_{1-\lambda} \circ \tau .
\end{equation}
This results in a certain ``asymmetry'' between the left and right (derived) module categories, which we will see in the next section. This does not happen for the $p$-DG nilHecke algebra since it also possesses an algebra automorphism ([Proposition 3.19, loc.~cit.]) that intertwines $(\NH,\dif_1)$ and $(\NH,\dif_{-1})$.
\end{rmk}

\section{Categorified Burau representation at a prime root of unity}
\label{catburau}

\paragraph{The NY resolution.}
Recall from the previous section that we have equipped the zig-zag algebra $A_n^!$ with a differential $\dif_\lambda$, $\lambda=0,1$. Since $A_n^!$ is positive, the simple left module $ L_i $ may be endowed with a trivial differential, making it a simple $p$-DG module over $ (A_n^!, \dif_\lambda)$.  Likewise the right module $_iL$ also has a right $p$-DG module structure.
Below we find finite cell modules quasi-isomorphic to $L_i$ and $_iL$ for $\lambda=0,1$.
Using the automorphism $\tau$ we need to only construct the cell modules for one of the parameters.

\begin{lemma} \label{lemma-ny-resolution}
\begin{enumerate}
\item[(i)] For $ \lambda=0$ and $1\leq i \leq n-1$, the left $p$-DG module $ L_i $ is quasi-isomorphic to the finite cell module
\begin{equation}\label{pdgresLi0}
\begin{gathered}
\xymatrix@C=0.7em{& & P_{i-1}\{3-2p \}  \ar@{=}[r] &\cdots\ar@{=}[r]& P_{i-1} \{ -3 \} \ar@{=}[r] & P_{i-1} \{ -1 \} \ar[drr]^{(i-1|i)} &  \\
P_i\{2-2p \} \ar[urr]^{(i|i-1)} \ar[drr]_{-(i|i+1)} & & & &  &  &  &  P_i ~ ,\\
& & P_{i+1} \{3-2p \}  \ar@{=}[r] \ar[uu]^{(i+1|i|i-1)} & \cdots \ar@{=}[r]  & P_{i+1} \{-3 \}  \ar@{=}[r] \ar[uu]^{(i+1|i|i-1)} &  P_{i+1} \{-1 \}  \ar[urr]_{(i+1|i)} \ar[uu]^{(i+1|i|i-1)} & &}
\end{gathered}
\end{equation}
while the right $p$-DG module ${_i L}$ is quasi-isomorphic to
\begin{equation}\label{pdgresLi1}
\begin{gathered}
\xymatrix@C=0.7em{& & {}_{i-1} P \{3-2p \} \ar[dd]_{(i+1|i|i-1)}  \ar@{=}[r] &\cdots\ar@{=}[r]& {}_{i-1} P \{ -3 \} \ar[dd]_{(i+1|i|i-1)} \ar@{=}[r] & {}_{i-1} P \{ -1 \} \ar[dd]_{(i+1|i|i-1)} \ar[drr]^{(i|i-1)} &  \\
{}_i P\{2-2p \} \ar[urr]^{(i-1|i)} \ar[drr]_{-(i+1|i)} & & & &  &  &  & {}_i P~.\\
& & {}_{i+1} P \{3-2p \}  \ar@{=}[r]  & \cdots \ar@{=}[r]  & {}_{i+1} P \{-3 \}  \ar@{=}[r]  &  {}_{i+1} P \{-1 \}  \ar[urr]_{(i|i+1)}  & &}
\end{gathered}
\end{equation}

\item[(ii)] For $ \lambda=1$ and $1\leq i \leq n-1$, the left $p$-DG module $L_i$ is quasi-isomorphic to the finite cell module
\begin{equation}\label{pdgresLi1}
\begin{gathered}
\xymatrix@C=0.7em{& & P_{i-1}\{3-2p \} \ar[dd]_{(i-1|i|i+1)}  \ar@{=}[r] &\cdots\ar@{=}[r]& P_{i-1} \{ -3 \} \ar[dd]_{(i-1|i|i+1)} \ar@{=}[r] & P_{i-1} \{ -1 \} \ar[dd]_{(i-1|i|i+1)} \ar[drr]^{(i-1|i)} &  \\
P_i\{2-2p \} \ar[urr]^{(i|i-1)} \ar[drr]_{-(i|i+1)} & & & &  &  &  &  P_i~,\\
& & P_{i+1} \{3-2p \}  \ar@{=}[r]  & \cdots \ar@{=}[r]  & P_{i+1} \{-3 \}  \ar@{=}[r]  &  P_{i+1} \{-1 \}  \ar[urr]_{(i+1|i)}  & &}
\end{gathered}
\end{equation}
while the right $p$-DG module $ {}_i L $ is quasi-isomorphic to
\begin{equation}\label{pdgresLi0}
\begin{gathered}
\xymatrix@C=0.7em{& & {}_{i-1}P \{3-2p \}  \ar@{=}[r] &\cdots\ar@{=}[r]& {}_{i-1}P \{ -3 \} \ar@{=}[r] & {}_{i-1}P \{ -1 \} \ar[drr]^{(i|i-1)} &  \\
{}_i P \{2-2p \} \ar[urr]^{(i-1|i)} \ar[drr]_{-(i+1|i)} & & & &  &  &  &  {}_i P~.\\
& & {}_{i+1} P \{3-2p \}  \ar@{=}[r] \ar[uu]^{(i-1|i|i+1)} & \cdots \ar@{=}[r]  & {}_{i+1} P \{-3 \}  \ar@{=}[r] \ar[uu]^{(i-1|i|i+1)} &  {}_{i+1} P \{-1 \}  \ar[urr]_{(i|i+1)} \ar[uu]^{(i-1|i|i+1)} & &}
\end{gathered}
\end{equation}
\end{enumerate}
\end{lemma}

Remark that in the lemma, we have adopted the convention that the $p$-DG modules $P_0:=0$ and $_0P:=0$, so that any map into or out of them vanishes.

\begin{proof}
We will only establish this for $ \lambda=1 $ and the left module $L_i$. The right module case is left as an exercise for the reader. The parameter $\lambda=0$ follows from applying the anti-automorphism $\tau$ as in equation \eqref{eqn-parameter-intertwiner}. For the ease of notation, we denote by $c_i=(i|i+1|i)$ the ``circle'' path based at the vertex $i$.

Consider the following finite cell $p$-DG module over $ A_n^!) $ (using Notation \ref{ntn-finite-cell-mod})
$$\xymatrix{P_{i-1}\{1\} \ar[rr]^{(i-1|i|i+1)} && P_{i+1}\{1\}},$$
whose differential is temporarily written as $ \widehat{\dif} $ to avoid confusion.
Observe that $\widehat{\dif}^p \equiv 0$ holds on this module since
$$\dif^k(i-1|i|i+1)=(k+1)!(i-1|i|i+1)c_{i+1}^k,$$
so that $\widehat{\dif}^p(i-1)=\dif^{p-1}(i-1|i|i+1)\equiv 0$ (c.f.~Example \ref{eg-finite-cofibrant-mod}).

Now we have a short exact sequence of $ {(A_n^!)}_\dif $-modules
\begin{equation}\label{pdgsimpleres}
\begin{gathered}
\xymatrix{  & & P_{i-1} \{1 \}\ar@{-->}[dd]^{_{(i-1|i|i+1)}}\ar[dr]^-{\psi} & &\\
0 \ar[r] & P_i \{2 \} \ar[ur]^-{\phi}  \ar[dr]_-{\phi} && P_i \ar[r] & 0~,\\
&& P_{i+1} \{ 1 \}\ar[ur]_-{\psi} && }
\end{gathered}
\end{equation}
where
\[
\psi=
\left(
  \begin{array}{ccc}
  (i-1|i)  &  (i+1|i)   \\
  \end{array}
\right)
\hspace{.5in}
\phi =
\left(
  \begin{array}{ccc}
  { (i|i-1)} \\
  {- (i|i+1)}   \\
  \end{array}
\right) ,
\]
and the dashed arrow indicates that it is an internal differential of the mid-column module equipped with $\widehat{\dif}$.
We show that the maps $ \phi $ and $\psi$ commute with differentials.
On the one hand $\phi \partial (i) = 0$, while on the other hand
\begin{equation*}
\widehat{\partial} \phi (i) = \widehat{\partial} ((i|i-1)-(i|i+1)) = (i|i-1|i|i+1) - (i|i+1|i|i+1) = 0.
\end{equation*}
Similarly, we check
\begin{gather*}
\psi \widehat{\partial} (i+1) = 0= \partial (i+1|i)= \partial \psi (i+1),
\\
\psi \widehat{\partial} (i-1) = \psi (i-1|i|i+1) = (i-1|i|i+1|i)= \partial (i-1|i)=\partial \psi (i-1) .
\end{gather*}


We then obtain the desired finite cell resolution by applying Lemma~\ref{lemma-ses-gives-p-DG-resolution} to the short exact sequence~\eqref{pdgsimpleres}. The degree shifts are made according to Remark \ref{rmk-deg-shifts}.
\end{proof}

In what follows, the cofibrant replacement constructed above will be referred to as the ``NY resolution''. We gather here some immediate consequences arising from having this resolution.

\begin{cor}\label{cor-RHOM-of-simples}For $\lambda=0,1$ and $1\leq i \leq n-1$, there are homotopy equivalences of $p$-complexes
\[
\RHOM_{A_n^!}(L_i,L_j)\cong
\left\{
\begin{array}{cc}
\widetilde{V}_0\oplus \widetilde{V}_0\{2p-2\}\cong \widetilde{V}_0\oplus \widetilde{V}_0\{-2\}[-2]& j=i,\\
\widetilde{V}_{p-2}\{p-1\} \cong \widetilde{V}_0\{-1\}[-1] & j=i\pm 1,\\
0 & |j-i|>1.
\end{array}
\right.
\]
\end{cor}
\begin{proof}We show how to derive the first isomorphism from Lemma \ref{lemma-ny-resolution} and leave the reader to prove the rest. Denote by $\mathbf{p}(L_i)$ the cofibrant replacement constructed in Lemma. Then there are homotopy equivalences of $p$-complexes
\begin{align*}
\RHOM_{A_n^!}(L_i,L_i)  & \cong \HOM_{A_n^!}(\mathbf{p}(L_i),L_i) \\
& \cong \HOM_{A_n^!}(P_i,L_i) \oplus \HOM_{A_n^!}(P_i\{2-2p\},L_i)
\cong \Bbbk \oplus \Bbbk\{2p-2\},
\end{align*}
where we have used that $\HOM_{A_n^!}(P_j,L_i)\cong (j)L_i\cong \delta_{ij}\Bbbk$. The result follows.
\end{proof}

In the proof of the next result, we will use the following fact. If $A$ is a positive $p$-DG algebra, and
$$\xymatrix{\cdots \ar[r] & P_i\{r\} \ar[r]^-{a}& P_j\{s\}\ar[r] & \cdots}$$
is part of a finite cell left module $P$, where $a$ is an element in $\HOM_{A}(P_i,P_j)=e_iAe_j$, then the dual module $\RHOM_{A}(P,A)$, equipped with the induced differential (negating the differential on $A$), is a finite cell right $A$-module, equipped with the conjugate filtrations and opposite grading shifts:
$$\xymatrix{\cdots \ar[r] & {_jP}\{-s\} \ar[r]^-{-a}& {_iP}\{-r\}\ar[r] & \cdots.}$$
This follows easily from the way in which the Hopf algebra $H=\Bbbk[\dif]/(\dif^p)$ acts on dual representations.

\begin{cor}\label{cor-dual-mod-iso-right-mod}For $\lambda=0,1$ and $i\in \{1,\dots, n-1\}$, there is an isomorphism of right (resp.~left) $p$-DG modules in the derived category $\mc{D}(A_n^!,\dif_\lambda)$
\[
\RHOM_{A_n^!}(L_i, A_n^!) \cong {_iL}[-2]\{-2\} \quad \quad \left(\textrm{resp.}~\RHOM_{A_n^!}({_iL}, A_n^!)\cong L_i[-2]\{-2\} \right).
\]
\end{cor}

\begin{proof}We show the claim for $\lambda=1$.  By the identification
$$\RHOM_{A_n^!}(L_i,A_n^!)\cong \HOM_{A_n^!}(\mathbf{p}(L_i),A_n^!),$$
where $\mathbf{p}(L_i)$ is the resolution in Lemma \ref{lemma-ny-resolution}, and the remark above, we may identify $\RHOM_{A_n^!}(L_i,A_n^!)$ with the finite cell module
\[
\begin{gathered}
\xymatrix@C=0.9em{& & {}_{i-1}P \{1\}  \ar[r]^-{-1} & {}_{i-1}P \{ 3 \} \ar[r]^-{-1}&\cdots\ar[r]^-{-1} & {}_{i-1}P \{ 2p-3 \} \ar[drr]^-{-(i|i-1)} &  \\
{}_i P  \ar[urr]^-{-(i-1|i)} \ar[drr]_-{(i+1|i)} & & & &  &  &  &  {}_i P\{2p-2\}~.\\
& & {}_{i+1} P \{1 \}  \ar[r]^-{-1} \ar[uu]_-{_{-(i-1|i|i+1)}} & {}_{i+1} P \{3 \}  \ar[r]^-{-1} \ar[uu]_-{_{-(i-1|i|i+1)}} & \cdots \ar[r]^-{-1}  &  {}_{i+1} P \{2p-3 \}  \ar[urr]_{-(i|i+1)} \ar[uu]_-{_{-(i-1|i|i+1)}} & &}
\end{gathered}
\]
Comparing this with the resolution for $_iL$ one concludes that the isomorphism above exists, up to an overall minus sign.
\end{proof}

\begin{rmk}
Although we do not need a cofibrant resolution of $ L_n $ to construct a categorical braid group action we include it here for completeness since it is easy for $ \lambda=0$.
\begin{itemize}
\item[(i)] As a left $p$-DG module for $ \lambda=0$, $ L_n $ is quasi-isomorphic to the finite cell module
\[
\xymatrix{P_{n-1}\{3-2p \}  \ar@{=}[r] &\cdots\ar@{=}[r]& P_{n-1} \{ -3 \} \ar@{=}[r] & P_{n-1} \{ -1 \} \ar[rr]^-{(n-1|n)} & &  P_n.
}
\]

\item[(ii)] As a right $p$-DG module for $ \lambda=0$, $ {}_n L $ is quasi-isomorphic to the finite cell module
\[
\xymatrix{{}_{n-1} P \{3-2p \}  \ar@{=}[r] &\cdots\ar@{=}[r]& {}_{n-1} P \{ -3 \} \ar@{=}[r] & {}_{n-1} P \{ -1 \} \ar[rr]^-{(n|n-1)} & &  {}_n P.
}
\]
\end{itemize}
\end{rmk}

\begin{lemma}For $\lambda \in \{0,1\}$, the left (resp.~right) $p$-DG module $L_n$ (resp.~$_nL$) is compact in the derived category $\mc{D}(A_n^!,\dif_\lambda)$. Consequently, $(A_n^!,\dif_\lambda)$ is left (resp.~right) hopfologically finite.
\end{lemma}

\begin{proof}We only prove the result for the left module $L_n$ when $\lambda =1$. The rest of the cases are taken care of by the previous remark and \eqref{eqn-parameter-intertwiner}.

Consider the cofibrant module $P_1$ which is a $p$-DG direct summand of $A_n^!$. It has a basis consisting of elements
\[
\left\{(i|i-1|\dots|1)\big|1\leq i \leq n\right\}.
\]
It follows that $P_1$, as a module over the smash product ring $(A_n^!)_\dif$, has an increasing Jordan-H\"older filtration
\[
F^0=0, \quad F^i=\bigoplus_{j=1}^i \Bbbk (n-j+1|\dots|1),
\]
whose subquotients are isomorphic to the simples $L_i$ ($i=1,\dots, n$), with $F^1\cong L_n$ being the socle. It follows that, in the derived category, $P_1$ is a convolution of $L_n$ and the other simples $L_i$ ($i=1,\dots, n-1$), the latter ones being compact due to the NY resolutions. Since $P_1$ is compact, applying the standard ``two-out-of-three'' property, we see that $L_n$ must also be compact.
\end{proof}

\begin{rmk}\label{rmk-hermitian-structure}
As a summary of the discussion in this section, we conclude that the family of $p$-DG algebras $(A_n^!,\dif_\lambda)$, $\lambda\in \{0,1\}$, is both left and right hopofologically finite (Definition \ref{def-hopfological-finite-algebra}). Thus we may identify the finite and compact derived categories of either left or right $p$-DG modules over these algebras. Since the simple modules over $A_n^!$ are absolutely irreducible, the $\RHOM$-pairing on $\mc{D}^c(A_n^!)$ (c.f. \eqref{eqn-categorical-pairing} \eqref{eqn-K-group-pairing}) descends to a perfect semi-linear pairing on the Grothendieck group
\[
[\RHOM_{A_n^!}(-,-)] \colon K_0(A_n^!)\times K_0(A_n^!)\lra \mathbb{O}_{p},
\]
which is conjugate linear in the first variable and linear in the second variable. Base changing along the surjective ring map
\[
\mathbb{O}_p\lra \mathcal{O}_{2p}, \quad q\mapsto \zeta_{2p},
\]
one can embed $K_0(A_n^!)\o_{\mathbb{O}_p}\mathcal{O}_{2p}$ as a rank-$n$ $\mathcal{O}_{2p}$-lattice inside a genuine hermitian $\mathbb{C}$-vector space with the extended hermitian form. In what follows, we will abuse notation and refer to the above $\mathbb{O}_p$-structure as ``hermitian'', even though it is not a subring of $\mathbb{C}$.
\end{rmk}

\paragraph{Temperley-Lieb action.} Now we construct a categorical Temperley-Lieb algebra action on $\mc{D}^c(A_n^!,\dif_\lambda)$, where $\lambda = 0,1$.
\begin{defn}\label{defn-cup-cap-functors} Consider the $p$-DG algebra $(A_n^!,\dif_\lambda)$, where $\lambda\in \{0,1\}$, and let $1\leq i \leq n-1$.
\begin{enumerate}
\item[(i)] We define the \emph{$i$th cup functor} to be the functor
\[
\cup_i \colon \mathcal{D}^c(\Bbbk) \lra \mathcal{D}^c(A_n^!) , \quad \quad \cup_i(V) : = L_i \otimes V.
\]

\item[(ii)] The \emph{$i$th cap functor} is given by
\[
\cap_i \colon \mathcal{D}^c(A_n^!) \lra  \mathcal{D}^c(\Bbbk) , \quad \quad \cap_i(M) : = {}_i L \otimes^\mathbf{L}_{A_n^!} M.
\]
\end{enumerate}
\end{defn}

A useful property of these functors is that, up to some grading shifts, they are biadjoint to each other.

\begin{lemma}
\label{lemma-adjoints}
Let $ \lambda \in \{0,1 \} $ and $ i \in \{ 1,\ldots,n-1 \}$. The functor $\cup_i$ admits right and left adjoints.
\begin{enumerate}
\item The right adjoint of the functor $ \cup_i \colon \mc{D}^c(\Bbbk, \partial_{\lambda}) \longrightarrow \mc{D}^c(A_n^!, \partial_{\lambda}) $ is isomorphic to
$$ \cap_i[-2]\{-2\} \colon \mc{D}^c(A_n^!, \partial_{\lambda}) \longrightarrow \mc{D}^c(\Bbbk, \partial_{\lambda}) .$$
\item The left adjoint of the functor $ \cup_i \colon \mc{D}^c(\Bbbk, \partial_{\lambda}) \longrightarrow \mc{D}^c(A_n^!, \partial_{\lambda}) $ is isomorphic to
$$ \cap_i \colon \mc{D}^c(A_n^!, \partial_{\lambda}) \longrightarrow \mc{D}^c(\Bbbk, \partial_{\lambda}) $$
\end{enumerate}
\end{lemma}

\begin{proof}
First we compute the right adjoint of $ \cup_i $. Taking any $p$-complex $U$ and (left) $p$-DG module $M$ over $A_n^!$, we have
\begin{align}\label{eqn-RHOM-biadjoint-1}
\RHOM_{A_n^!}(L_i \otimes_{\Bbbk} U, M) &\cong \RHOM_{\Bbbk}(U, \RHOM_{A_n^!}(L_i,M))\nonumber \\
& \cong \RHOM_{\Bbbk}(U, \RHOM_{A_n^!}(L_i,A_n^!)\otimes_{A_n^!}^\mathbf{L} M)\nonumber \\
& \cong \RHOM_{\Bbbk}(U, {}_i L[-2] \{ -2 \} \otimes_{A_n^!}^\mathbf{L} M).
\end{align}
The first isomorphism follows from the tensor-hom adjunction in the $p$-DG setting (\cite[Lemma 8.14]{QYHopf}).
The second isomorphism is true by considering a cofibrant resolution of $M$ and reducing it to the case $M\cong A_n^!$, where it is then obvious.
The third isomorphism is given in Corollary ~\ref{cor-dual-mod-iso-right-mod}.

Next we compute the left adjoint of  $ \cup_i $.
\begin{align}\label{eqn-RHOM-biadjoint-2}
 \RHOM_{\Bbbk}({}_iL \otimes_{A_n^!} M, U) & \cong \RHOM_{A_n^!}(M, \RHOM_{\Bbbk}({}_iL_,\Bbbk) \otimes_{\Bbbk} U)\nonumber \\
& \cong \RHOM_{A_n^!}(M, L_i \otimes_{\Bbbk} U)
\end{align}
The first isomorphism is once again just the tensor-hom adjunction, while for the second one, we note that, as $p$-DG bimodules over $(A_n^!,\Bbbk)$, there is an isomorphism
$$ \RHOM_{\Bbbk}({}_iL_,\Bbbk)\cong ({_iL})^* \cong L_i, $$
where we have used that any finite dimensional $p$-complex is automatically finite-cell, and $\RHOM_\Bbbk(-,\Bbbk)$ is just taking the vector space dual. Now the lemma follows from taking degree-zero stable $H$-invariant (c.f.~\eqref{eqn-morphism-space-as-invariants}) on both sides of the above equations \eqref{eqn-RHOM-biadjoint-1}, \eqref{eqn-RHOM-biadjoint-2}.
\end{proof}

We will also write down the explicit unit and counit maps associated to the biadjunction. To do this, we need the following lemma, which is the tensor analogue of Corollary \ref{cor-RHOM-of-simples}.

\begin{lemma}\label{lemma-tensor-simples}For the $p$-DG module $L_i$ over $(A_n^!,\dif_\lambda)$, where $\lambda\in \{0,1\}$ and $1\leq i \leq n-1$, there are isomorphisms of $p$-complexes
\[
_jL\otimes_{A_n^!}^{\mathbf{L}} {L_i} \cong
\left\{
\begin{array}{cc}
\widetilde{V}_0[2] \{2 \} \oplus \widetilde{V}_0 & \textrm{if $j=i$,}\\
\widetilde{V}_{p-2}\{1-p\} \cong \widetilde{V}_0[1]\{1\} & \textrm{if $j=i\pm 1$,}\\
0 & \textrm{if $|j-i|>1$.}
\end{array}
\right.
\]
\end{lemma}

\begin{proof}
The proof is entirely similar to that of Corollary \ref{cor-RHOM-of-simples}. We leave it to the reader as an exercise.
\end{proof}

Now we compute the units and counits of the above biadjunction, which are given by maps of $p$-DG bimodules. We start with the adjoint pair $(\cup_i,\cap_i[-2]\{-2\})$. Then the counit and unit are respectively given by the composition of $p$-DG bimodule maps over $A_n^!$ and $\Bbbk$
\begin{gather}
\epsilon_1 \colon L_i\o {_iL}[-2]\{-2\}\cong L_i\o \RHOM_{A_n^!}(L_i,A_n^!) \stackrel{\lambda_1}{\lra} A_n^!, \label{eqn-biadjunction-1}\\
\eta_1 \colon \Bbbk \stackrel{\pi_1}{\lra}  {_iL}[-2]\{-2\} \o_{A_n^!}^{\mathbf{L}}L_i \cong  \Bbbk \oplus \Bbbk[-2]\{-2\}. \label{eqn-biadjunction-2}
\end{gather}
Here the first isomorphism in $\epsilon_1$ is given by Corollary \ref{cor-RHOM-of-simples}, and the map $\lambda_1$ is the natural pairing map.\footnote{To be more precise this pairing map is given by the natural pairing $\mathbf{p}(L_i)\o \HOM_{A_n^!}(\mathbf{p}(L_i), A_n^!)\lra A_n^!$, $(x,f)\mapsto f(x)$. It involves a choice of a cofibrant replacement, and thus is only well-defined in the derived category.} As for $\eta_1$, the isomorphism comes from the previous lemma, while $\pi_1$ is the inclusion map into the first factor.

For the adjoint pair $(\cap_i,\cup_i)$, we have
\begin{gather}
\epsilon_2 \colon _iL\o_{A_n^!}^{\mathbf{L}} {L_i}\cong \Bbbk[2]\{2\} \oplus \Bbbk \stackrel{\pi_2}{\lra}\Bbbk, \label{eqn-biadjunction-3}\\
\eta_2 \colon A_n^! \stackrel{\lambda_2}{\lra} \RHOM_\Bbbk (\RHOM_{A_n^!}(A_n^!,{L_i}),{L_i}) \cong L_i\o {_iL}. \label{eqn-biadjunction-4}
\end{gather}
The first isomorphism in $\epsilon_2$ comes from Lemma \ref{lemma-tensor-simples}, while $\pi_2$ is the projection onto the second factor. Similarly as for $\epsilon_1$, the map $\lambda_2$ in $\eta_2$ is a canonical evaluation map. Notice that in equation \eqref{eqn-biadjunction-4}, both $\RHOM$'s can be replaced by $\HOM$ since $A_n^!$ (resp.~any $p$-complex) is cofibrant over $A_n^!$ (resp.~$\Bbbk$). From this, one readily checks that $\eta_2$ is, up to a non-zero constant, the unique bimodule map which sends $(i)\in A_n^!$ to $(i)\o (i)\in L_i\o {_iL}$, and other idempotents to zero.

Now for each $i\in \{1,\dots, n-1\}$ and $\lambda\in \{0,1\}$, we define an endo-functor
\begin{equation}\label{eqn-def-TL-generator}
\mf{U}_i = \cup_i \circ \cap_i [-1]\lbrace -1 \rbrace  : \mc{D}^c(A_n^!,\dif_\lambda) \lra \mc{D}^c(A_n^!,\dif_\lambda),
\end{equation}
which is diagrammatically presented as
\[
\mf{U}_i=
\begin{DGCpicture}[scale=0.8]
\DGCstrand(-0.5,0)(-0.5,2)[$^1$`$\empty$]
\DGCstrand(0.25,0)(0.25,2)[$^2$`$\empty$]
\DGCstrand(1,0)(1,0.25)(2,0.25)(2,0)/d/[$^i$`$^{i+1}$]
\DGCstrand/d/(1,2)(1,1.75)(2,1.75)(2,2)
\DGCstrand(3.5,0)(3.5,2)[$^{n-1}$`$\empty$]
\DGCstrand(2.75,0)(2.75,2)[$^n$`$\empty$]
\DGCcoupon*(0.45,0.5)(0.95,1.5){$\dots$}
\DGCcoupon*(2.15,0.5)(2.65,1.5){$\dots$}
\end{DGCpicture}
\ .
\]
The diagrammatics explains the names ``cups'' and ``caps'' introduced earlier.

\begin{thm}
\label{thm-TL-action}
For $ i=1, \ldots, n-1 $, the functors $ \mf{U}_i $ are self-biadjoint, and they enjoy the following functor-isomorphism relations.
\begin{enumerate}
\item[(i)] $ \mf{U}_i \circ \mf{U}_i \cong \mf{U}_i [-1]\lbrace -1 \rbrace \oplus \mf{U}_i[1] \lbrace 1 \rbrace $,
\item[(ii)] $ \mf{U}_i \circ \mf{U}_j \cong \mf{U}_j \circ \mf{U}_i $ for $ |i-j|>1 $,
\item[(iii)] $ \mf{U}_i \circ \mf{U}_j \circ \mf{U}_i \cong \mf{U}_i $ for $ |i-j|=1 $.
\end{enumerate}
\end{thm}

The theorem tells us that there is a categorical $TL_n$, the $n$-strand Temperley-Lieb algebra, action on the derived category $\mc{D}^c(A_n^!,\dif_\lambda)$, $\lambda=0,1$. Before giving its proof, let us recall that $TL_n$ can be graphically depicted by locally generated string diagrams as above for $\mf{U}_i$, with the product structure
\[
\begin{DGCpicture}
\DGCstrand(0,0)(0,1.5)
\DGCstrand(0.5,0)(0.5,1.5)
\DGCstrand(1.5,0)(1.5,1.5)
\DGCcoupon*(0.6,0.1)(1.4,0.4){$\dots$}
\DGCcoupon*(0.6,1.1)(1.4,1.4){$\dots$}
\DGCcoupon(-0.1,0.5)(1.6,1){$x$}
\end{DGCpicture}
\ \cdot \
\begin{DGCpicture}
\DGCstrand(0,0)(0,1.5)
\DGCstrand(0.5,0)(0.5,1.5)
\DGCstrand(1.5,0)(1.5,1.5)
\DGCcoupon*(0.6,0.1)(1.4,0.4){$\dots$}
\DGCcoupon*(0.6,1.1)(1.4,1.4){$\dots$}
\DGCcoupon(-0.1,0.5)(1.6,1){$y$}
\end{DGCpicture}
\ = \
\begin{DGCpicture}[scale={1,0.75}]
\DGCstrand(0,0)(0,2)
\DGCstrand(0.5,0)(0.5,2)
\DGCstrand(1.5,0)(1.5,2)
\DGCcoupon*(0.6,0.1)(1.4,0.2){$\dots$}
\DGCcoupon*(0.6,0.9)(1.4,1.1){$\dots$}
\DGCcoupon*(0.6,1.8)(1.4,1.9){$\dots$}
\DGCcoupon(-0.1,0.3)(1.6,0.8){$y$}
\DGCcoupon(-0.1,1.2)(1.6,1.7){$x$}
\end{DGCpicture}
\ ,
\]
for any $x,y\in TL_n$. The defining relations are presented as local isotopy simplifications, and the evaluation of circle relation
\[
\begin{DGCpicture}
\DGCbubble(0,0){0.45}
\end{DGCpicture}
\ = \
-q -q^{-1}.
\]
On the Grothendieck group level, if we denote the symbols of the functors $\mf{U}_i$ by $u_i:=[\mf{U}_i]$, then $u_i$ are hermitian operators on $K_0(A_n^!,\dif_\lambda)$ with respect to the hermitian form
$$[\RHOM_{A_n^!}(-,-)]:K_0(A_n^!)\times K_0(A_n^!)\lra \mathbb{O}_p.$$

\begin{proof}
The biadjunction follows from Lemma \ref{lemma-adjoints}. The functor isomorphisms are direct consequences of Lemma \ref{lemma-tensor-simples}. For instance, we check $(i)$ as follows. For any $M\in \mc{D}^c(M,\dif_\lambda)$, there are functorial-in-$M$ isomorphisms
\begin{align*}
\mf{U}_i \circ \mf{U}_i (M) & \cong (L_i\o {_iL}[-1]\{-1\})\o_{A_n^!}^{\mathbf{L}}(L_i\o {_iL}[-1]\{-1\})\o_{A_n^!}^{\mathbf{L}} (M)\\
& \cong L_i\o ({_iL}\o_{A_n^!}^{\mathbf{L}}L_i[-2]\{-2\})\o ({_iL}\o_{A_n^!}^{\mathbf{L}} M)\\
& \cong L_i \o (\Bbbk[-2]\{-2\}\oplus \Bbbk)\o({_iL}\o_{A_n^!}^{\mathbf{L}} M)\\
& \cong (L_i \o \Bbbk[-2]\{-2\}\o{_iL})\o_{A_n^!}^{\mathbf{L}} M \oplus (L_i\o \Bbbk\o{_iL})\o_{A_n^!}^{\mathbf{L}} M\\
& \cong \mf{U}_i[-1]\{-1\}(M) \oplus \mf{U}_i[1]\{1\}(M).
\end{align*}
where in the third isomorphism, we have used Lemma \ref{lemma-tensor-simples}.
\end{proof}

\paragraph{Braid group action.} Now we construct a categorical braid group action on $\mc{D}(A_n^!,\dif_\lambda)$, where $\lambda \in \{0,1\}$. We will construct, for each braid group generators $t_i$, a $p$-complex of $p$-DG bimodule over $A_n^!$. These bimodules are $p$-DG analogues of the bimodules used by Khovanov-Seidel~\cite{KS} for the Koszul dual algebra. Alternatively, it is the analogue of Webster's braiding bimodule~\cite{Web4}, which exists in much greater generality, in the $p$-DG context. We will show that these bimodules, regarded as functors on the derived categories, satisfy the braid group relations.

\begin{defn}\label{def-twist-functors}For each $i\in \{1,\dots, n-1\}$, we define the functors $\mf{T}_i$ and $\mf{T}_i^\prime$ as follows.
\begin{enumerate}
\item[(i)] The functor $\mf{T}_i: \mc{D}^c(A_n^!)\lra \mc{D}^c(A_n^!)$ is given by the derived tensor product with the cocone of the bimodule adjunction map \eqref{eqn-biadjunction-4}
\[
 A_n^! \xrightarrow{\lambda_2[-1]} \RHOM_\Bbbk (\RHOM_{A_n^!}(A_n^!,{L_i}),{L_i}) [-1],
\]

\item[(ii)]The functor $\mf{T}_i^\prime: \mc{D}^c(A_n^!)\lra \mc{D}^c(A_n^!)$ is given by the derived tensor product with the cone of the bimodule adjunction map \eqref{eqn-biadjunction-1}:
\[
L_i \o \RHOM_{A_n^!}(L_i,A_n^!)[1] \xrightarrow{\lambda_1[-1]} A_n^!,
\]
\end{enumerate}
\end{defn}

More explicitly, given any $M\in \mc{D}^c(A_n^!)$, we have filtered $p$-DG modules over $A_n^!$
\begin{align*}
\mf{T}_i(M) & = M \xrightarrow{\lambda_2[-1]} \RHOM_\Bbbk (\RHOM_{A_n^!}(M,{L_i}),{L_i}) [-1]\\
& = \left( M \xrightarrow{-\lambda_2} \RHOM_\Bbbk (\RHOM_{A_n^!}(M,{L_i}),{L_i})=\cdots =\RHOM_\Bbbk (\RHOM_{A_n^!}(M,{L_i}),{L_i}) \right),
\end{align*}
and
\begin{align*}
\mf{T}^\prime_i(M) & =L_i \o \RHOM_{A_n^!}(L_i,M)[1] \xrightarrow{\lambda_1[-1]} M\\
 & = \left( L_i \o \RHOM_{A_n^!}(L_i,M)=\cdots = L_i \o \RHOM_{A_n^!}(L_i,M) \xrightarrow{-\lambda_1} M \right).
\end{align*}
See Lemma \ref{lemma-cone-construction} and the discussion after it on the notation adopted here.

\begin{prop}\label{prop-biadjoint-braiding}The functors $\mf{T}_i$ and $\mf{T}_i^\prime$ are inverses of each other. In particular, they are biadjoint functors.
\end{prop}

\begin{proof}For computational convenience, we identify the $p$-DG bimodules
\begin{gather*}
\RHOM_\Bbbk (\RHOM_{A_n^!}(A_n^!,{L_i}),{L_i})\cong L_i\o {_iL}, \\
L_i \o \RHOM_{A_n^!}(L_i,A_n^!)\cong \mathbf{p}(L_i)\o {_iL[-2]\{-2\}},
\end{gather*}
and we compute the tensor product of functors as tensor product of bimodules. For instance, $\mf{T}_i\circ \mf{T}_i^\prime$ is represented by
collapsing the cube
\[
\xymatrix{
(L_i\o {_iL})[-1]\o_{A_n^!}(\mathbf{p}(L_i)\o {_iL[-2]\{-2\}})[1] \ar[rr]^-{\Id\o \lambda_1[-1]}&& (L_i\o {_iL})[-1]\o_{A_n^!}A_n^!\\
A_n^!\o_{A_n^!}(\mathbf{p}(L_i)\o {_iL[-2]\{-2\}})[1]\ar[u]^{\lambda_2[-1]\o \Id}\ar[rr]^-{\Id\o \lambda_1[-1]}&& A_n^!\o_{A_n^!}A_n^!~.\ar[u]_{\lambda_2[-1]\o\Id}
}
\]
along the southwest-northeast direction into a single $p$-complex of bimodules. Using Lemma \ref{lemma-tensor-simples}, we identify the terms on the left vertical arrow with
\begin{gather}
(L_i\o {_iL})[-1]\o_{A_n^!}(\mathbf{p}(L_i)\o {_iL[-2]\{-2\}})[1] \cong (L_i\o {_iL})\oplus (L_i\o {_iL})[-2]\{-2\},\label{eqn-term1}\\
(A_n^!\o_{A_n^!}\mathbf{p}(L_i)\o {_iL[-2]\{-2\}})[1] \cong L_i\o {_iL}[-1]\{-2\}\label{eqn-term2}
\end{gather}
Via a chasing of the definitions, one recognizes that part of the top arrow, generated by the second summand on the right of \eqref{eqn-term1}, as being quasi-isomorphic to the contractible complex of $p$-DG bimodules,
\[
L_i\o {_iL} \xrightarrow{\Id[-1]} L_i\o {_iL}[-1].
\]
Modulo this subcomplex, the cube contains $A_n^!$ as a sub-bimodule sitting in the southeast corner, which lies in $q$-degree zero. Further quotienting out this copy of $A_n^!$, the complex of bimodules leftover comes from the quotient of the left vertical arrow, and it is quasi-isomorphic to the contractible cocone
\[
L_i\o {_iL}[-1]\{-2\} \xrightarrow{\Id[-1]} L_i[-2]\{-2\}.
\]
It follows that the composition functor is quasi-isomorphic to $\mf{T}_i\circ \mf{T}_i^\prime \cong A_n^!\o_{A_n^!}(-)\cong \Id$. The composition in the other way follows by reflecting the above cube diagram about the southwest-northeast diagonal. The last claim is then clear.
\end{proof}

\begin{prop}\label{prop-RIII-braiding}
The derived functors $ \mf{T}_i $ satisfy the Reidemeister-III braid relations.  That is, for $i=1,\ldots,n-1$, there are quasi-isomorphisms of filtered $p$-DG bimodules $ \mf{T}_i \mf{T}_{i\pm1} \mf{T}_i \cong \mf{T}_{i\pm1} \mf{T}_{i} \mf{T}_{i\pm1} $.
\end{prop}

\begin{proof}
Throughout the course of this proof we will write $ A $ for $ A_n^!$ and simply write the derived tensor symbol for $\o_{A_n^!}^{\mathbf{L}}$ as $\o_A^{\mathbf{L}}$.
Recall that $ \mf{T}_i$, by construction, is the tensor product functor with the filtered $p$-DG bimodule $  A \rightarrow L_i \otimes {}_i L [-1]$.
We identify the composition $ \mf{T}_i \mf{T}_{i+1} \mf{T}_i $ as the (derived) tensor-product functor with the filtered $p$-DG bimodule
\begin{equation}
\label{tensorcube}
\xymatrix{
X_{0} \ar[r]^{d_{0}} & X_{1} \ar[r]^{d_{1}} & X_{2} \ar[r]^{d_{2}} & X_3 ,
}
\end{equation}
where $ d_{0} \colon X_{0} \rightarrow X_{1} $ is given by
\begin{equation*}
\xymatrix{
&& L_i \otimes {}_i L \otimes_A A \otimes_A A [-1]  \\
A \otimes_A A \otimes_A A \ar[rr] \ar[urr] \ar[drr] && A \otimes_A L_{i+1} \otimes {}_{i+1} L \otimes_A A[-1]  , \\
&& A \otimes_A A \otimes_A  L_i \otimes {}_i L [-1]
}
\end{equation*}
$ d_1 \colon X_1 \rightarrow X_2 $ is given by
\begin{equation*}
\xymatrix{
L_i \otimes {}_i L \otimes_A A \otimes_A A [-1] \ar[r] \ar[dr] & L_i \otimes {}_i L \otimes_A L_{i+1} \otimes {}_{i+1} L \otimes_A A [-2]  \\
A \otimes_A L_{i+1} \otimes {}_{i+1} L \otimes_A A[-1] \ar[r] \ar[ur] \ar[dr]& L_i \otimes {}_i L \otimes_A A \otimes_A L_i \otimes {}_i L[-2]   ,\\
A \otimes_A A \otimes_A  L_i \otimes {}_i L [-1] \ar[r] \ar[ur] & A \otimes_A L_{i+1} \otimes {}_{i+1} L \otimes_A  L_i \otimes {}_i L [-2]
}
\end{equation*}
$ d_2 \colon X_2 \rightarrow X_3 $ is given by
\begin{equation*}
\xymatrix{
 L_i \otimes {}_i L \otimes_A^\mathbf{L} L_{i+1} \otimes {}_{i+1} L \otimes_A A [-2] \ar[dr]  & \\
L_i \otimes {}_i L \otimes_A A \otimes_A L_i \otimes {}_i L[-2]\ar[r] & L_i \otimes {}_i L \otimes_A^\mathbf{L} L_{i+1} \otimes {}_{i+1} L \otimes_A^\mathbf{L} L_i \otimes {}_i L [-3]. \\
A \otimes_A L_{i+1} \otimes {}_{i+1} L \otimes_A^\mathbf{L} L_i \otimes {}_i L [-2]  \ar[ur] &
}
\end{equation*}
The restriction
\begin{equation*}
d_2 \colon L_i \otimes {}_i L \otimes_A A \otimes_A L_i \otimes {}_i L[-2] \rightarrow L_i \otimes {}_i L \otimes_A^\mathbf{L} L_i \otimes {}_i L \otimes_A^\mathbf{L} L_i \otimes {}_i L [-3]
\end{equation*}
arises as tensoring the (co)cone of the dotted arrow in the top row of the diagram
\begin{equation}
\label{boundary5to7}
\begin{gathered}
\xymatrix @C=0.95em{
 {}_i L \otimes_A A \otimes_A L_i \ar[d]^{\cong} \ar@{.>}[r]& {}_i L \otimes_A^\mathbf{L} L_{i+1} \otimes {}_{i+1} L \otimes_A^\mathbf{L} L_i  &\\
  {\bf p}({}_i L) \otimes_A A \otimes_A {\bf p}(L_i) \ar[r]  & {\bf p}({}_i L) \otimes_A L_{i+1} \otimes {}_{i+1} L \otimes_A {\bf p}(L_i)  \ar[u]^{\cong} \ar[r]^-{ \cong} & \Bbbk[1]\{1\} \otimes \Bbbk[1]\{1\} \ar[d]^-{\cong}\\
 {\Bbbk\oplus \Bbbk[2]\{2 \}}\ar[u]^-{\cong} \ar[rr]^{\pi} &&\Bbbk[2] \{ 2\}
}
\end{gathered}
\end{equation}
on the left and on the right by $ {L_i} $ and $ {{}_iL} $ respectively.

We now show that the map of $p$-complexes $\pi$ is non-zero.  Then for degree reasons it must be the projection onto the second factor.  To do this, we recall that the resolutions $ {\bf p}(L_i) $ and $ {\bf p}({}_i L) $ fit into the short exact sequences of $ A_{\partial}$-modules
\begin{equation*}
\xymatrix{
0 \ar[r] & (P_{i+1}[1] \{1\} \rightarrow P_i) \ar[r] & {\bf p}(L_i) \ar[r] & (P_{i}[2] \{2\} \rightarrow P_{i-1}[1]\{1\})\ar[r] & 0
}
\end{equation*}
\begin{equation*}
\xymatrix{
0 \ar[r] & ({}_{i-1} P [1] \{1\} \rightarrow {}_i P) \ar[r] & {\bf p}({}_i L) \ar[r] & ({}_i P [2] \{2\} \rightarrow {}_{i+1} P [1]\{1\}) \ar[r] & 0.
}
\end{equation*}
There are also short exact sequences
\begin{equation*}
\xymatrix{
0 \ar[r] & (P_{i+2} \rightarrow P_{i+1}[-1]\{-1\}) \ar[r] & {\bf p}(L_{i+1})[-1]\{-1\} \ar[r] & (P_{i+1}[1] \{1\} \rightarrow P_{i})\ar[r] & 0
}
\end{equation*}
\begin{equation*}
\xymatrix{
0 \ar[r] & ({}_{i} P [2] \{2\} \rightarrow {}_{i+1} P[1]\{1\}) \ar[r] & {\bf p}({}_{i+1} L)[1]\{1\} \ar[r] & ({}_{i+1} P [3] \{3\} \rightarrow {}_{i+2} P [2]\{2\}) \ar[r] & 0.
}
\end{equation*}
These short exact sequences give rise to a commutative diagram
\begin{equation}
\label{commtriangle1}
\begin{gathered}
\xymatrix{
{\bf p}(L_{i+1})[-1]\{-1\} \ar[rr] \ar[dr] & & {\bf p}(L_i) \\
& (P_{i+1}[1] \{1\} \rightarrow P_i) \ar[ur] &
}
\end{gathered}
\end{equation}
Tensoring ~\eqref{commtriangle1} on the left by $ {}_{i+1} L $ produces the commutative triangle in ~\eqref{commtriangle2}
\begin{equation}
\label{commtriangle2}
\begin{gathered}
\xymatrix @C=0.85pc{
{\begin{matrix} \Bbbk(i+1) \otimes (i+1)[-1]\{-1\} \\ \oplus \\ \Bbbk(i+1) \otimes (i+1)[1]\{1\} \end{matrix}} \ar[rr] & &  \Bbbk(i+1) \otimes (i+1)[1]\{1\} \\
{}_{i+1} L \otimes_A  {\bf p}(L_{i+1})[-1]\{-1\} \ar[rr]^-{\rho_1} \ar[dr] \ar[u]^-{\cong} & & {}_{i+1} L \otimes_A {\bf p}(L_i) \ar[u]^-{\cong} \\
& {\Bbbk}(i+1) \otimes (i+1) [1]\{1\}   \ar[ur] &
}
\end{gathered}
\end{equation}
where the map in the top row is projection.
Tensoring ~\eqref{commtriangle1} on the left by $ {}_{i} L $ produces the commutative triangle in ~\eqref{commtriangle3}
\begin{equation}
\label{commtriangle3}
\begin{gathered}
\xymatrix{
\Bbbk(i) \otimes (i) \ar[rr] & & {\begin{matrix} \Bbbk(i) \otimes (i) \\ \oplus \\ \Bbbk(i) \otimes (i)[2]\{2\} \end{matrix}} \\
{}_{i} L \otimes_A  {\bf p}(L_{i+1})[-1]\{-1\} \ar[rr]^-{\rho_2} \ar[dr] \ar[u]^-{\cong} & & {}_{i} L \otimes_A {\bf p}(L_i) \ar[u]_-{\cong } \\
& {\Bbbk}(i) \otimes (i)    \ar[ur] &
}
\end{gathered}
\end{equation}
where the map in the top row is now inclusion. In particular, we deduce that $\rho_1$, $\rho_2$ are non-zero maps, and they will also be non-zero if we replace $_{i+1}L$ or $_iL$ by their resolutions.

Now expanding ~\eqref{boundary5to7} we see the map is non-zero:

\begin{equation*}
\begin{gathered}
\xymatrix@C=.7em{
 {\bf p}({}_i L) \otimes_A A \otimes_A {\bf p}(L_i)  \ar[r] & {\bf p}({}_i L) \otimes_A L_{i+1} \otimes {}_{i+1} L \otimes_A {\bf p}(L_i) \\
  {\bf p}({}_i L) \otimes_A A \otimes_A {\bf p}(L_{i+1}) [-1] \{-1\} \ar[r] \ar[u]^{\rho_2 \neq 0} & {\bf p}({}_i L) \otimes_A L_{i+1} \otimes {}_{i+1} L \otimes_A {\bf p}(L_{i+1}) [-1]\{-1\} \ar[u]^{\Id \o \rho_1 \neq 0} \\
  {\bf p}({}_i L) \otimes_A A \otimes_A L_{i+1} [-1] \{-1\} \ar[u]^{\cong} & \Bbbk(i+1) \otimes (i+1) \otimes (i+1) \otimes (i+1) \ar[u]^{\cong} \\
 \Bbbk (i+1) \otimes (i+1) \otimes (i+1) \ar[ur]^{\neq 0} \ar[u]^{\cong}. &
}
\end{gathered}
\end{equation*}
Here the slanted arrow is non-zero since the bimodule map $A\lra L_{i+1}\o {_{i+1}L}$ given in \eqref{eqn-biadjunction-4} sends the idempotent $(i+1)\in A$ to $(i+1)\o (i+1)\in L_{i+1}\o {_{i+1}L}$.
Consequently, after replacing $\mathbf{p}(L_i),~\mathbf{p}({_iL})$ by the simples, the map of $p$-DG bimodules has the effect
\begin{equation*}
{}_i L \otimes_A A \otimes_A L_i \rightarrow {}_i L \otimes_A^\mathbf{L} L_{i+1} \otimes {}_{i+1} L \otimes_A^\mathbf{L} L_i \cong {\bf p}({}_i L) \otimes_A L_{i+1} \otimes {}_{i+1} L \otimes_A {\bf p}(L_i)
\end{equation*}
\begin{equation*}
(i) \otimes 1 \otimes (i) [2] \{ 2\} \mapsto (i+1) \otimes (i+1) [1]\{1\} \otimes (i+1) \otimes (i+1)[1]\{1\}.
\end{equation*}
It follows that the subcomplex $ d_2 \colon X_2 \rightarrow X_3 $ of ~\eqref{tensorcube} contains an acyclic subcomplex of $p$-DG bimodules
$L_i \otimes {}_i L\{ 2\} \stackrel{[-1]}{\lra} L_i \otimes {}_i L\{ 2\} [-1]$:
\begin{equation*}
\xymatrix{
L_i \otimes {}_i L \otimes_A A \otimes_A L_i \otimes {}_i L [-2] \ar[r]^{[-1] \hspace{.4in}} \ar[d]^{\cong} & L_i \otimes {}_i L \otimes_A^\mathbf{L} L_{i+1} \otimes {}_{i+1} L \otimes_A^\mathbf{L} L_i \otimes {}_i L [-2] \ar[d]^{\cong} \\
L_i \otimes {}_i L[-2] \oplus L_i \otimes {}_i L\{ 2\} \ar[r]^-{( 0 ,~ [-1])} & L_i \otimes {}_i L\{ 2\} [-1].
}
\end{equation*}
Modulo this acyclic term, we have another subcomplex of $p$-DG bimodules:
\begin{equation}
\label{symmetricpart}
\begin{gathered}
\xymatrix{
& L_i \otimes {}_i L[-1] \ar[r] \ar[ddr] & L_i \otimes {}_{i+1} L[-1]\{1\} \\
A \ar[ur] \ar[dr] & & \\
& L_{i+1} \otimes {}_{i+1} L [-1] \ar[r] \ar[uur] & L_{i+1} \otimes {}_{i} L [-1] \{1\}
}
\end{gathered}
\end{equation}
where $ L_i \otimes {}_i L[-1] $ is a diagonal submodule of the first and third terms of $ X_1$.  Note that this diagonal does not map into
$ L_i \otimes {}_i L \otimes_A A \otimes_A L_i \otimes {}_i L[-2] $ of $ X_2 $ since $ (i) \otimes (i)[-1] \in L_i \otimes {}_i L [-1] $ comes from $ A $ and with its preimage in $ A $ it already forms a $p$-DG submodule.
Quotienting by this subcomplex, we get another acyclic complex of bimodules
\begin{equation*}
\xymatrix{
L_i \otimes {}_i L [-1] \ar[r]^{[-1]} & L_i \otimes {}_i L [-2].
}
\end{equation*}
Since ~\eqref{symmetricpart} is symmetric in $i$ and $i+1$ we get
$ \mf{T}_i \mf{T}_{i+1} \mf{T}_i \cong \mf{T}_{i+1} \mf{T}_{i} \mf{T}_{i+1} $.
\end{proof}

We now state our main result.

\begin{thm}
\label{thm-braidrelations}
The functors $ \mf{T}_i, \mf{T}_i' $ for $ i=1, \ldots, n-1 $ satisfy braid relations, i.e., there are isomorphisms of functors
\begin{enumerate}
\item $ \mf{T}_i \mf{T}_i' \cong \Id \cong \mf{T}_i' \mf{T}_i $
\item $ \mf{T}_i  \mf{T}_j \cong \mf{T}_j \mf{T}_i $ if $ |i-j|>1 $
\item $ \mf{T}_i \mf{T}_{i+1} \mf{T}_i \cong \mf{T}_{i+1} \mf{T}_{i} \mf{T}_{i+1} $.
\end{enumerate}
\end{thm}

\begin{proof} The first and last isomorphisms are proven in Proposition \ref{prop-biadjoint-braiding} and Proposition \ref{prop-RIII-braiding}. The second isomorphism is easy, and we leave it as an exercise to the reader.
\end{proof}

\begin{prop}
In the Grothendieck group we have equalities:
\begin{equation*}
[\mf{T}_i]=[\Id]-q^{p+1}[\mf{U}_i]
\end{equation*}
\begin{equation*}
[\mf{T}_i']=[\Id]-q^{p-1}[\mf{U}_i].
\end{equation*}
\end{prop}

\begin{proof}
By the definition of $ \mf{U}_i $, the definitions of $ \mf{T}_i $ and $ \mf{T}_i' $ may be reformulated as:
\begin{equation}
\label{defT_i&&}
\mf{T}_i \cong A_n^! \rightarrow \mf{U}_i[1] \{3 \} \rightarrow \mf{U}_i [1] \{5 \} \rightarrow \cdots \rightarrow \mf{U}_i [1] \{2p-1 \}
\end{equation}
\begin{equation}
\label{defT_i'&&}
\mf{T}_i' \cong \mf{U}_i [1] \{1 \} \rightarrow \cdots \mf{U}_i[1] \{2p-5 \} \rightarrow \mf{U}_i[1] \{2p-3 \} \rightarrow A_n^!.
\end{equation}
The proposition follows by recalling that $ 1+q^2+ \cdots + q^{2p-2}=0 $.
\end{proof}

\begin{prop}
The categorical braid group action descends to the Burau representation in the Grothendieck group.  More specifically, the following squares commute
\begin{equation*}
\xymatrix{
K_0(A_n^!) \ar[r]^{\cong \hspace{.1in}} \ar[d]^{[\mf{T}_i]} & V_1^{\otimes n}[n-2] \ar[d]^{t_i} \\
K_0(A_n^!) \ar[r]^{\cong \hspace{.1in}} & V_1^{\otimes n} [n-2]
}
\hspace{1in}
\xymatrix{
K_0(A_n^!) \ar[r]^{\cong \hspace{.1in}} \ar[d]^{[\mf{T}_i']} & V_1^{\otimes n}[n-2] \ar[d]^{t_i'} \\
K_0(A_n^!) \ar[r]^{\cong \hspace{.1in}} & V_1^{\otimes n} [n-2].
}
\end{equation*}
\end{prop}

\begin{proof}
One applies the functors $ \mf{T}_i $ and $ \mf{T}_i'$ to simple objects $ L_j $ using Corollary ~\ref{cor-RHOM-of-simples}.  It is an easy exercise to check that, in the Grothendieck group, these computations match up with the formulas from
~\eqref{twistingdualcan1} and ~\eqref{twistingdualcan2}.
\end{proof}

\appendix

\section{Maps between simples}

For $ \lambda=1$ the unique non-trivial map (up to homotopy) $ \psi_{i+1} \colon L_{i+1}[1] \lbrace 2p-1 \rbrace \rightarrow L_i $ is easy to construct.
The map between their cofibrant resolutions
\begin{equation*}
\xymatrix@C=1em{
& & & & P_{i+2} \ar^{(i+2|i+1)}[drr] & & & & \\
P_{i+1} \{-2p+3 \} \ar@{=}[r] &\cdots \ar@{=}[r] & P_{i+1} \{-1 \} \ar^{-(i+1|i+2)}[rru] \ar_{(i+1|i)}[rrd] & & & & P_{i+1} \{1 \} \ar@{=}[r] & \cdots \ar@{=}[r]  & P_{i+1} \{2p-3 \} \\
& & & & P_{i}  \ar_{(i|i+1)} [urr] \ar_{_{(i|i+1|i+2)}}[uu] & & & &
}
\end{equation*}
and
\begin{equation*}
\xymatrix@C=0.6em{
& & P_{i-1}\{-(2p-3) \} \ar[dd]^{(i-1|i|i+1)}  \ar@{=}[r] &\cdots\ar@{=}[r]& P_{i-1} \{ -1 \} \ar[dd]^{(i-1|i|i+1)} \ar[drrrr]^{(i-1|i)} &  \\
P_i\{-(2p-2) \} \ar[urr]^{(i|i-1)} \ar[drr]_{-(i|i+1)} & & &   &  &  &  & &  P_i\\
& & P_{i+1} \{-(2p-3) \}  \ar@{=}[r]  & \cdots \ar@{=}[r]   &  P_{i+1} \{-1 \}  \ar[urrrr]_{(i+1|i)}  & &
}
\end{equation*}
respectively is given by the identity map $ P_{i+1} \{ -2p+3+2r \} \rightarrow P_{i+1} \{-2p+3+2r \} $ for $ r=0,\ldots,p-2 $ and is zero otherwise.

For $ \lambda=1$ the unique (up to homotopy) map $ \phi_i \colon L_i \rightarrow L_{i+1}[1] \{ 1 \} $ is more difficult to construct.
For $ p=2$ the map is given by:
\begin{equation}
\xymatrix@C=4em{
 & P_{i-1} \{ -1 \} \ar[dd]^{(i-1|i|i+1)} \ar[dr]^{(i-1|i)} &  \\
P_i\{-2 \} \ar[ur]^{(i|i-1)} \ar[dr]^{(i|i+1)} \ar[ddr]_{(i)} & &  P_i \ar[ddd]^{(i|i+1)} \ar@/^2.0pc/@[red][ldddd]|(.43){(i|i+1|i+2) }\\
& P_{i+1} \{-1 \} \ar[ur]^{(i+1|i)} \ar[rdd]|(.3){(i+1) }  \ar@/^2.0pc/@[red][ddd]|(.7){(i+1|i+2)}& & \\
 & P_{i} \{ -2 \} \ar[dd]_{(i|i+1|i+2)} \ar[dr]_(.5){(i|i+1)} &  \\
P_{i+1}\{-3 \} \ar[ur]_{(i+1|i)} \ar[dr]^{(i+1|i+2)} & &  P_{i+1} \{ -1 \} \\
& P_{i+2} \{-2 \} \ar[ur]_{(i+2|i+1)}   & & \\
}
\end{equation}

For a prime $ p>2$ there is a map $ \phi_i \colon L_i \rightarrow L_{i+1}[1] \{ 1 \} $
between the cofibrant resolution of $ L_i$
\begin{equation*}
\xymatrix@C=0.2em{
& & P_{i-1}\{-(2p-3) \} \ar[dd]^{(i-1|i|i+1)}  \ar@{=}[r] &\cdots\ar@{=}[r]& P_{i-1} \{ -1 \} \ar[dd]^{(i-1|i|i+1)} \ar[drrrr]^{(i-1|i)} &  \\
P_i\{-(2p-2) \} \ar[urr]^{(i|i-1)} \ar[drr]_-{-(i|i+1)} & & & &  &  & &  &  P_i\\
& & P_{i+1} \{-(2p-3) \}  \ar@{=}[r]  & \cdots \ar@{=}[r]  &  P_{i+1} \{-1 \}  \ar[urrrr]_{(i+1|i)}  & &
}
\end{equation*}
and the cofibrant resolution of $ L_{i+1}[1] \{ 1 \} $
\begin{equation*}
\xymatrix@C=0.2em{
& & & P_{i+2} \{ -2p+2 \} \ar^-{(i+2|i+1)}[dr] & & & \\
P_{i+1} \{-4p+5 \} \ar@{=}[r]  &\cdots \ar@{=}[r] & P_{i+1} \{-2p+1 \} \ar^-{-(i+1|i+2)}[ru] \ar_{(i+1|i)}[rd]  & & P_{i+1} \{-2p+3 \} \ar@{=}[r] & \cdots \ar@{=}[r]  & P_{i+1} \{-1\} \\
& & & P_{i} \{ -2p+2 \} \ar_{(i|i+1)} [ur] \ar^{(i|i+1|i+2)}[uu] & & & \\
}
\end{equation*}
given on components as follows:
\begin{itemize}
\item $ d_j \colon P_i \rightarrow P_{i+1} \{-2p+3+2j \} $ for $ j=0,\ldots,p-2$
\newline $ d_j = \cdot (-1)^{j+1} (p-j-2)!(i|i+1)(i+1|i|i+1)^{p-j-2} $
\item $ b \colon P_i \rightarrow P_{i+2} \{ -2p+2 \} $
\newline $ b = \cdot -(i|i+1|i+2)(i+2|i+1| i+2)^{p-2} $
\item $ p_{k,j} \colon P_{i+1} \{ -1-2k \} \rightarrow P_{i+1} \{ -2p+3+2j \} $ for $ 0 \leq j \leq p-3-k $
\newline $ p_{k,j} = \cdot (-1)^{k+j} \frac{(p-2-j)!}{(k+1)!} (i+1|i|i+1)^{p-j-k-2} $
\item $ p_{k,p-2-k} \colon P_{i+1} \{ -1-2k \} \rightarrow P_{i+1} \{ -2k-1 \} $ for $ 0 \leq k \leq p-2 $
\newline $ p_{k,p-2-k} = -1$
\item $ h_{k,j} \colon P_{i-1} \{-1-2k \} \rightarrow P_{i+1} \{ -2p+3+2j \} $ for $ 0 \leq j \leq p-3-k $
\newline $ h_{k,j} = \cdot \frac{(-1)^{k+1} k}{(k+1)!(j+1)!} (i-1|i|i+1)(i+1|i|i+1)^{p-j-k-3} $
\item $ m_k \colon P_{i+1} \{-1-2k \} \rightarrow P_{i+2} \{-2p+2 \}$ for $ 0 \leq k \leq p-3$
\newline $ m_k=\cdot \frac{(-1)^k}{(k+1)!} (i+1|i+2)(i+2|i+1|i+2)^{p-k-2} $
\item $ m_{p-2} \colon P_{i+1} \{-2p+3 \} \rightarrow P_{i+2} \{-2p+2 \} $
\newline $ m_{p-2} = -(i+1|i+2) $
\item $ n_{p-2} \colon P_{i+1} \{-2p+3 \} \rightarrow P_{i} \{-2p+2 \} $
\newline $ n_{p-2} = 2(i+1|i)$
\item $ f_k \colon P_{i-1} \{ -1-2k \} \rightarrow P_{i+2} \{-2p+2 \} $ for $ 1 \leq k \leq p-3 $
\newline $ f_k = \cdot \frac{(-1)^{k+1} k}{(k+1)!} (i-1|i|i+1|i+2)(i+2|i+1|i+2)^{p-k-3} $
\item $ g_{p-2} \colon P_{i-1} \{-2p+3\} \rightarrow P_i \{-2p+2\} $
\newline $ g_{p-2}= 2(i-1|i) $
\item $ \gamma \colon P_i \{-2p+2 \} \rightarrow P_i \{-2p+2 \} $
\newline $ \gamma=1 $.
\end{itemize}

\bibliographystyle{alpha}

%

\vspace{0.1in}

\noindent Y.~Q.: { \sl \small Department of Mathematics, University of California, Berkeley, Berkeley, CA 94720, USA} \newline \noindent {\tt \small email: yq2121@math.berkeley.edu}

\vspace{0.1in}

\noindent J.~S.: {\sl \small Department of Mathematics, CUNY Medgar Evers, Brooklyn, NY, 11225, USA}
\newline\noindent  {\tt \small email: joshuasussan@gmail.com}

%
\end{document}